\documentclass[10pt,twocolumn,twoside]{IEEEtran}
\usepackage[noadjust]{cite}

\usepackage{amsmath,amsthm,mathrsfs, verbatim,color,lscape,xspace}
\usepackage{url}

\usepackage{listings}
 \usepackage{graphicx}

 \usepackage{hyperref}
\usepackage{amsmath, dsfont}
\usepackage{amsfonts}
\usepackage{amssymb}
\usepackage{enumerate}
\usepackage{ifpdf}
\usepackage{lscape}

\usepackage[dvipsnames]{xcolor}
\usepackage{booktabs}
\usepackage{ifthen}

\usepackage{subcaption}

\newcommand{\iu}{\boldsymbol{\mathrm{i}}}

\newcommand{\ra}{\rightarrow}

\newcommand{\da}{\downarrow}
\newcommand{\prob}[1]{P\left(#1\right)}

\newcommand{\R}{\mathbb{R}}

\DeclareMathOperator*{\argmin}{arg\,min}
\newcommand{\norm}[1]{\| #1 \|_{\infty}}

\newcommand{\IP}{\mathcal{I}_{\text{pow}}}
\newcommand{\IPi}{\mathcal{I}_{\text{pow},i}}
\newcommand{\IC}{\mathcal{I}_{\text{cur}}}
\newcommand{\ICS}{\mathcal{I}_{\text{cur}^2}}

\newcommand{\IT}{\mathcal{I}_{\text{tmp},\tau}}

\newcommand{\WH}{K_i}

\newcommand{\pr}{^{\prime}}
\newcommand{\prpr}{^{\prime\prime}}
\newcommand{\eps}{\epsilon}
\newcommand{\pa}{\frac{\partial}{\partial\tau_{\ell}}}
\newcommand{\rk}{\text{rank }}
\newcommand{\diag}{\text{diag}}

\newcommand{\Rdet}{\mathcal{R}_{\text{det}}}

\newcommand{\Rcur}{\mathcal{\tilde{R}}^{(\text{cur})}_{\eps,p}}
\newcommand{\Rtemp}{\mathcal{\tilde{R}}^{(\text{tmp},\tau)}_{\eps,p}}
\newcommand{\Rlb}{\mathcal{\tilde{R}}^{(\text{tmp},\tau,LB)}_{\eps,p}}
\newcommand{\Rtaylor}{\mathcal{\tilde{R}}^{(\text{tmp},\tau,TL)}_{\eps,p}}

\newtheorem{lemma}{Lemma}[section]

\newtheorem{proposition}{Proposition}[section]

\newtheorem{Taylor approximation}{Taylor approximation}
\numberwithin{equation}{section}

\makeatletter
\renewcommand*\env@matrix[1][*\c@MaxMatrixCols c]{%
  \hskip -\arraycolsep
  \let\@ifnextchar\new@ifnextchar
  \array{#1}}
\makeatother

\makeatletter
\newboolean{showcomments}
\setboolean{showcomments}{true}
\newcommand{\tom}[1]{  \ifthenelse{\boolean{showcomments}}
{ \textcolor{orange}{(TN says:  #1)}} {}  }
\newcommand{\jk}[1]{  \ifthenelse{\boolean{showcomments}}
{ \textcolor{green}{(JK says:  #1)}} {}  }
\newcommand{\bert}[1]{\ifthenelse{\boolean{showcomments}}
{ \textcolor{red}{(BZ says:  #1)}}{}}
\newcommand{\todo}[1]{\ifthenelse{\boolean{showcomments}}
{ \textcolor{green}{(To do:  #1)}}{}}

\newcommand{\ignore}[1]{}

\begin{document}
\title{Temperature Overloads in Power Grids Under Uncertainty: a Large Deviations Approach}

\author{Tommaso Nesti, Jayakrishnan Nair and Bert Zwart, \textit{Member, IEEE}

\thanks{
This research is supported by an NWO VICI grant.

T. Nesti is with Centrum
Wiskunde  Informatica (CWI), Science Park 123, Amsterdam, 1098 XG, Netherlands (email: nesti@cwi.nl).

Jayakrishnan Nair is with IIT Bombay, Powai, Mumbai, Maharashtra 400076, India (jayakrishnan.nair@ee.iitb.ac.in).

B. Zwart is with CWI, and also with Eindhoven University of Technology (email: bert.zwart@cwi.nl).

}}

\maketitle
\begin{abstract}
The  advent  of  renewable  energy  has  huge  implications  for the design and control of power grids.
 Due to increasing supply-side uncertainty, traditional reliability constraints such as strict bounds on current,
 voltage and temperature in a transmission line have to be replaced by computationally demanding chance constraints.
In this paper we use large deviations techniques to study the probability of current and temperature
overloads in power grids with stochastic power injections, and develop corresponding safe
capacity regions. In particular, we characterize the set of admissible power injections such that the probability of overloading of any line over a given time interval stays below a fixed target. We show how enforcing (stochastic) constraints on temperature, rather than on current, results in a less conservative approach and can thus lead to capacity gains.
\end{abstract}
\section{Introduction}

The electricity network is one of the backbones of modern society, and
is expected to function at all times. The advent of renewable energy
sources such as wind and solar generation have put this requirement under
pressure due to their considerable intermittency. Both the US and
Europe have set long-term goals on the usage of renewable energy, but
the effects of the integration of renewable sources into the power grid
are already felt today. For example, 80\% of the
bottlenecks in the European transmission grid are already caused by
renewables \cite{Yao2015}. Dealing with the uncertainty of renewable
generation effectively is therefore an essential requirement in the operation of
modern grids.

A well-controlled power grid matches supply and demand at all times,
ensuring that line constraints are not violated. The system operator
achieves this by making periodic control actions (typically every 5-15
minutes) that adapt the operating point of the grid in response to
changing conditions \cite{Ela2011}. A key assumption driving grid
operation today is that the grid remains roughly static between
control instants. In other words, it is assumed that the operating
point does not change much between control instants. Thus, the
operator simply ensures that line constraints are satisfied at each
control instant. This assumption is of course reasonable when there is
little short term uncertainty in demand and supply.

However, with increasing penetration of renewable sources, supply-side
uncertainty is bound to grow dramatically going forward. Renewable
energy sources, like wind and solar, can exhibit considerable
variability in power generation in the short term
\cite{Mills2010,Hodge2012}. Consequently, in the near future, system operators
will no longer be able to assume that the grid is static between
control instants, and will have to set the operating point taking into
account its variability in the short-term. This entails setting the
operating point of the grid with stochastic guarantees on constraint
satisfaction \cite{Fu2001,Wadman2012}. In other words, the operating
point must be set such that line constraint violation is a
sufficiently rare event until the next control instant. Moreover,
schemes like Optimal Power Flow (OPF) need to be
adapted in such a way that uncertainty is taken into account, and
outages stay rare events.

In an optimization framework, this leads to chance constraints which
are hard to evaluate analytically. The analysis of such constraints,
such as the probability of overheating or a blackout, is often done
using rare event simulation techniques
\cite{Wadman2012,Shortle2013}. Although detailed
simulations can be more accurate, short-term planning requires tools
that enable the grid operator to handle the stochastic constraints
much faster.

The main contribution of this paper is the development of tractable
capacity regions for a power grid with variable sources. Specifically,
we characterize the set of admissible power injections such that the
overheating of any transmission line over a given interval is a rare
event.  For the simplest network with two nodes and one line, our
results have been published without proof in the extended abstract
\cite{Bosman2014}.

Our main technique to achieve this is the theory of large
deviations. Specifically, we model the random power input sources as
small-noise stochastic differential equations (SDE), for
which a comprehensive and sufficiently explicit theory of large
deviations is available.  SDEs are a flexible modeling tool
  for continuously varying processes, and their use for wind speed
  modeling has been adopted recently by several
  authors~\cite{Wadman2016,Iversen2016,Moller2016,Iversen2017}. We model power flows on
the network using the DC approximation, which is standard in the
literature of high-voltage transmission system analysis
\cite{Purchala2005,Powell2004,Stott2009}.
(More realistic models based on AC power flow are often analytically
intractable, and may not even be well-posed
\cite{Iwamoto1981,Baillieul1982,Molzahn2013}.)
This allows us to apply Freidlin-Wentzell theory~\cite{Dembo1998} to
approximate the probability of an overload event, which in turn leads
to our capacity region characterization.

Avoiding transmission line overheating is a key reliability
constraint in order to avoid sag and loss of tensile strength
\cite{Wan1999}, one of the key causes of the Northeast blackout in $2003$  \cite{Can2004} and the San Diego blackout
in $2011$ \cite{Cal2012}. The classical approach for enforcing this constraint is to impose a certain upper bound on each line current. In Section~\ref{sc: current}, we follow this approach and develop capacity regions based on bounding the probability that any line current exceeds its limit
over a given interval. We prove an important convexity property of
this capacity region, which enables its application in optimization
formulations such as OPF. When the random power
injections are modeled by an Ornstein-Uhlenbeck (OU) process, we
express this capacity region in closed form.

Since line temperature responds gradually to current, a transient current overload does not
necessarily imply a temperature overload. Thus, imposing a
constraint on the probability of current overload results in a smaller
capacity region compared to the same constraint on the probability of
temperature overload. This observation was noted via simulations in
\cite{Wadman2012}. We show that it
is possible to develop large deviations estimates for temperature
constraint violations that lead to larger capacity regions, than the
ones obtained when only considering currents. To the best of our
knowledge, this paper provides the first analytical treatment of this
phenomenon.

 As it turns out, it is hard to compute such regions. We overcome
this issue by developing two tractable approximations: the first is an inner bound, and the second is based on a
Taylor series expansion of the decay rate of the temperature overload
probability. Both of the two regions coming out of these approximations capture the benefits of incorporating the
transient relationship between temperature and current. Moreover, they both
have the same computational complexity as the current-based
capacity region. For the case of OU power injections, we
express these capacity regions in closed form.

There are several related strands of literature, apart from the papers
dealing with rare event simulation that have been mentioned above.
Much of the literature on power flow in electricity grids considers
deterministic settings, focusing on computational and/or optimization
issues. Power flow papers that analyze stochastic models include
\cite{Wan1999,Fu2001,Phan2014,Summers2014}.

One remark about these papers is that they model stochastic behavior at particular snapshots
of time, as opposed to the process-level model in this
paper. Process-level models have been considered in simulation studies
\cite{Wadman2012} and in recent works on
chance-constrained versions of OPF
\cite{Bienstock2013,Bienstock2014}.
Other papers on chance-constrained OPFs include \cite{Vrakopoulou2013_2}, where the authors integrate probabilistic guarantees in a DC OPF via a scenario approach, and recent works on chance-constrained AC OPF, which tackle the nonlinearities by means of convex approximations~\cite{Bolognani2016}, relaxations~\cite{Vrakopoulou2013} and local linearization around a forecasted operating point~\cite{Roald2017}.

In this paper, we propose to approximate chance-constraints by using large-deviations (LD) theory. This is markedly different from the approach
of using Monte-Carlo methods~\cite{Calafiore2005}. Such methods can be effective but can require a large number of samples. Specifically, they are difficult to implement in our dynamic continuous-time setting. 
Another approach is to develop analytic convex approximations along the lines suggested in~\cite{Nemirovski2006}. Unfortunately, this method also seems mainly designed for static problems, and it appears hard to implement it in our dynamic setting.

In this work, we aim to develop analytic
tools which are explicit enough so as to be useful for planning and control of power grids in the short-term. Our work is complementary to recent efforts on
managing supply-side uncertainty via demand response
\cite{li11,Madaeni12}, energy storage \cite{Barton04,kim2009optimal},
and market (re)design \cite{Weber2010,nair2014}. A recent work on the analysis of temperature constraints in a discrete-time setting is~\cite{Kersulis2015}.

The paper is organized as follows. In Section~\ref{sc: model}, we
describe our model for power injections, line currents and line
temperatures. Sections~\ref{sc: current} and~\ref{sc: temperature}
constitute the core of the paper: we develop and characterize large
deviations-based capacity regions for line currents and line
temperatures, respectively, and provide explicit expressions in the
particular case that the power injections follow a multivariate
OU process. In Section~\ref{sc: numerics}, numerics
for the OU case are presented. We summarize and discuss future
directions in Section~\ref{sc: conclusions}.  Proofs are reported in
the Appendix.

\section{System Model} \label{sc: model}

\subsection{Model for the power grid and DC approximation}

The network is specified by a connected graph $\mathcal{G}=
(\mathcal{N} ,\mathcal{L} ),$ where $\mathcal{N} = \{0,1,2,\cdots,N\}$
is the set of nodes, modelling buses, and $\mathcal{L}$ is the set of
edges, representing the transmission lines. We have
$|\mathcal{N}|=N+1$ and $|\mathcal{L}|=L$.  After choosing an
arbitrary but fixed orientation of the transmission lines, we denote
by $\ell=(i,j)\in \mathcal{L} $ the transmission line between buses
$i$ and $j$, and by $w_{\ell}=w_{i,j}=w_{j,i}>0$ the weight of edge
$\ell=(i,j)$, corresponding to the \textit{susceptance} of that
transmission line. By convention, if there is no line between $i$ and
$j$ we set $w_{i,j}=w_{j,i}=0$.  The network structure is
described by the \textit{edge-vertex incidence matrix}
$A\in\R^{L\times (N+1)}$ defined as
  \begin{equation*}
 	A_{\ell, i}=\begin{cases}
    	1		&\text{if } \ell=(i,j),\\
    	-1		&\text{if } \ell=(j,i),\\
    	0    	&\text{otherwise}.
    \end{cases}
\end{equation*}
Denote by $W$ the $L \times L$ diagonal matrix defined as $W:=\mathrm{diag}(w_1,\dots, w_L)$.  The network topology and
weights are encoded in the \textit{weighted Laplacian matrix} of the
graph $G$, defined as $B := A^{\top} W A$ or entry-wise as
\begin{equation*}
	B_{i,j} =\begin{cases}
		-w_{i,j}	& \text{if } i \neq j,\\
		\sum_{k\neq i} w_{i,k} 	& \text{if } i=j.
	\end{cases}
\end{equation*}


%
%

Let $S(t)=(S_i(t))_{i \in \mathcal{N}}$ denote the vector of active
net power injections at time $t$, with the convention that $S_i(t) \ge
0$ ($S_i(t) \le 0$) means that power is generated (consumed,
respectively) at bus $i$. Node $0$ models the slack bus, which ensures
that there are no active power imbalances in the network.

Let $I(t)=(I_{\ell}(t))_{\ell \in \mathcal{L}}$ be the vector of line
currents, and $K(t)=(K_{\ell}(t))_{\ell \in \mathcal{L}}$ be the
vector of line temperatures. Each transmission line $\ell$ is
associated with a thermal limit $K_{\ell,\max},$ which is the maximum
permissible temperature of the line \cite{Wan1999}. We define
$I_{\ell,\max}> 0$ such that if $|I_{\ell}(t)|=I_{\ell,\max}$ at all
times, then $\lim_{t\to\infty} K_{\ell}(t)=K_{\ell,\max}.$ Throughout
this paper, we work with \textit{normalized} currents
$Y(t)=(Y_{\ell}(t))_{\ell \in \mathcal{L}}$, defined as
$Y_\ell(t)=I_{\ell}(t)/I_{\ell,\max}$.

In order to model the relation between power injections and line
currents, we make use of the \emph{DC approximation} (a.k.a. \emph{DC
  power flow},~\cite{Purchala2005}),
which leads to a linear relationship between power injections and
normalized line currents of the form $Y(t) = \overline{C} S(t)$, where
$\overline{C} \in \R^{L \times (N+1)}$ is a matrix encoding the
network topology and weights (Eq.~\eqref{eq:C blocks}).


In the remainder of this section, we briefly recall the DC
approximation. For notational simplicity, we suppress the dependence
of power, voltage and current on time when not essential. Let
$V_j=|V_j|e^{\iu\theta_j}$ denote the voltage at node $j$, with $V_j$
the voltage magnitude and $\theta_j$ the voltage phase.
The DC approximation consists of three assumptions: (i) voltage
magnitudes $|V_j|$ are all equal to $1$ (in the \textit{per-unit}
system), (ii) the phase differences between neighboring nodes are
small: $\forall\,(i, j)\in\mathcal{L}$, $|\theta_i - \theta_j|\ll 1$,
and (iii) the resistances of transmission lines are negligible with
respect to the reactances.
Under these assumptions, the reactive power flows are negligible
compared to the active power power flows. Moreover, the active power
flow $F_{i,j}$ on line $(i,j)\in\mathcal{L}$ can approximated
as $F_{i,j}\approx w_{i,j} (\theta_i - \theta_j).$ This allows
us to express the vector of (active) node power injections
$S\in\R^{N+1}$ in matrix form as
\begin{equation}\label{eq: DC power flow}
  S=B\theta,
\end{equation}
where $\theta\in\R^{N+1}$ is the vector of phase angles. Note that
$\sum_{i \in \mathcal{N}} S_i = 0,$ which is consistent with
Assumption~(iii) above ignoring line losses. Under the DC approximation, one can also approximate the line currents with the (active) power flow on the line $I_{i,j} \approx w_{i,j}(\theta_i - \theta_j)$.
%
The above
equivalence between the currents and power flows under this
approximation has been noted before (see, for
example,~\cite{Coffrin2012}).
Expressing the vector of line currents in matrix form, we obtain
\begin{equation}\label{eq: DC current}
 I=WA \theta.
\end{equation}

Following the derivation in \cite{Bienstock2014}, without loss of
generality we can set the phase angle at the slack bus equal to zero,
i.e. $\theta_0=0$, and rewrite the the system $S=B\theta$ as
\begin{equation}\label{eq: inverse}
\theta=\breve{B} S,
\end{equation}
where $\breve{B}=\begin{pmatrix}
            0 & 0\\
             0           & \hat{B}^{-1}
           \end{pmatrix}$ and
$\hat{B}$ is the $N$ by $N$ sub-matrix obtained
from $B$ by deleting the first row and first column.

Recall that the normalized currents are defined as
$Y_\ell(t)=I_{\ell}(t)/I_{\ell,\max}$, and let
$\Lambda=\mathrm{diag}(1/I_{1,\max},\dots, 1/I_{L,\max})$ .  In view of
Eq.~\eqref{eq: DC current} and \eqref{eq: inverse}, the active
normalized line currents $Y$ can be written as a linear
transformation of the power injections $S$,
i.e. $Y(t)=\overline{C} S(t)$, where $ \overline{C}:=\Lambda WA\breve{B}$.
\subsection{Stochastic and deterministic power injections}
We assume that power injections at nodes $1,\ldots,m\le N$ are
stochastic, modelling buses housing intermittent renewable power
generation.  On the other hand, power injections at nodes
$m+1,\ldots,N$ are assumed to be deterministic and constant, modeling
conventional loads/generators.

We will be interested in capturing the
probability of current/temperature overload over a finite horizon
$[0,T]$, which corresponds to the interval between periodic control
actions by the grid operator.
Thus, the buses in $\{m+1,\ldots,N\}$ are those that
may be assumed to have a steady power injection over this time
scale, denoted by $\mu_D$. Note that the power injection at the slack node $0$ is also
stochastic, since $S_0(t)=-\sum_{i=1}^n S_i(t)$.
 The power injection vector is of the form $S(t)
 =(S_0(t),X(t),\mu_D),$ where $X(t) \in \R^{m}$ is the vector of
 stochastic injections, and $\mu_D = (\mu_{D,i})_{i =
   m+1}^{N}\in\mathbb{R}^{N-m}$. We denote the initial condition for
 the stochastic power injections by $\mu := X(0)$, and let $\bar{\mu}
 := (\mu,\mu_D)$.

 In order to make the dependency of the normalized current on
 stochastic and deterministic power injections more explicit, we note
 that we can write
\begin{equation}\label{eq:C blocks}
Y(t)=\overline{C}S(t)=
 \begin{bmatrix}[c|c|c] & &     \\
       \mathbf{0} & \,\,C\,\, & \,\,C_D\,\,\\
       & &  \end{bmatrix}\begin{bmatrix}S_0(t)\\X(t)\\\mu_D\end{bmatrix},
\end{equation}
where $\mathbf{0}=[0,\ldots,0]^{\top}\in\mathbb{R}^L$, and $C\in \R^{L\times m},C_D\in \R^{L\times (n-m)}$ are the submatrices of $\overline{C}$ corresponding to stochastic and deterministic injections, respectively. More compactly,
\begin{equation}
  \label{eq: DC norm current exp}
  Y(t)=CX(t)+y,
\end{equation}
where $y:=C_D\mu_D$. We will refer to Eq.~\eqref{eq: DC norm
  current exp} as the DC power flow equations. The
following lemma shows that matrix $C$ has rank $m$, i.e. the number of stochastic power injections.
\begin{lemma}\label{lm: rank C}
  If the network graph is connected, $\rk(\overline{C})=N$ and
  $\rk(C)=m$. In particular, the matrix $C$ has linearly independent
  columns.
\end{lemma}

We may interpret $\bar{\mu}=(\mu,\mu_D)$ as the
vector of power injections set by the grid operator at time $0$ (for
example, $\bar{\mu}$ could be the result of an OPF planning).  Recall
that the initial condition for the normalized currents is
$Y(0) = \nu,$ where $\nu:= C\mu + y$. We are interested in scenarios
where power grids operate safely, by assuming that the nominal power
injections $\bar{\mu}$ are such that the corresponding expected line
currents at time $t=0$ do not exceed the critical level,
i.e. $\|\nu\|_{\infty}=\max_{\ell=1,\ldots,L}|\nu_{\ell}|\le 1$
(possibly, several $|\nu_{\ell}|$ could be close or equal to their threshold,
modeling a high-stress situation.
Subsequently, some of the power injections fluctuate randomly because of the variability of the
renewable generators. Our focus is to characterize the set of power injection vectors $\bar{\mu}$ such that the probability of current/temperature
overload over a finite horizon $[0,T]$ is below a prescribed target
$p.$
\subsection{Mapping between line current and line temperature}
In this section, we describe how line temperature depends on line
current. Recall that $K_{\ell}(t)$ denotes the temperature of line
$\ell.$ We work with \emph{normalized} line temperatures, defined as
follows. Let $K_{\text{env},\ell}$ be the ambient temperature
around line $\ell.$ We define the normalized line temperatures $\Theta(t)
= (\Theta_{\ell}(t))_{\ell \in \mathcal{L}}$ as
$\Theta_{\ell}(t)=\frac{K_{\ell}(t)-K_{\text{env},\ell}}{K_{\max,\ell}-K_{\text{env},\ell}}.$
Note that the reliability constraint on line temperatures reads
$\norm{\Theta_{\ell}} \le 1,$
where $\norm{f} := \max_{t \in [0,T]} \norm{f(t)}$ for a continuous function $f: [0,T] \ra \R^L$.

In this spirit, in
Section~\ref{sc: temperature} we characterize the capacity region of the power grid
based on bounding the temperature overload probability. In other
words, we describe the set of initial power injection vectors
$\bar{\mu}$ such that $\mathbb{P}(\norm{\Theta} \ge 1) \leq p,$ where $p$ is
a prescribed reliability target.

The transient relationship between the normalized temperature
$\Theta_{\ell}$ and the normalized current is given by the ordinary
differential equation~\cite{Pender1949}
\begin{equation}\label{eq: ode}
  \tau_{\ell} \frac{d\Theta_{\ell}}{dt}+\Theta_{\ell}=(Y_{\ell})^2,
\end{equation}
where $\tau_{\ell}>0$ denotes the thermal constant of the
transmission line $l$. Thus, we have
\begin{equation}\label{eq: xi}
  \Theta_{\ell}(t)=\Theta_{\ell}(0) e^{-t/\tau_{\ell}}+\frac{1}{\tau_{\ell}}\int_{0}^te^{-(t-s)/\tau_{\ell}}(Y_{\ell}(s))^2ds.
\end{equation}
Note that the instantaneous line temperature depends on the history of
the line current process, with an exponentially decaying weight on
past values. The parameter $\tau_{\ell}$ determines the dependence of
the instantaneous temperature on past values of current. If $\tau_{\ell}$ is small, the dependence on past current values becomes
weaker, i.e., the line temperature responds more quickly to changes in
current. In the limit as $\tau_{\ell} \da 0,$the response is instantaneous, i.e.,
$\Theta_{\ell}(t)=(Y_{\ell}(t))^2.$

For simplicity, we assume the initial condition
$\Theta_{\ell}(0)=(Y_{\ell}(0))^2=\nu_{\ell}^2\quad\forall \ell \in
\mathcal{L}$ for line temperatures. Note that $\nu_{\ell}^2$ is the
steady state temperature corresponding to a constant line current
$\nu_{\ell}.$\footnote{This assumption, which ignores the
    history of the temperature process prior to time $\approx 0$, is a
    natural engineering assumption if the actual line temperatures at
    time $\approx0$ are unavailable. If such measurements are
    available, it is possible to incorporate these into our capacity
    region based on temperature overload
    (Section~\ref{sc:temperature}) as well as its inner bound
    (Section~\ref{sc:temp_inner_bound}), although the analysis gets
    more complicated (see \cite[Section $4.3$]{Wadman2016}).} With
the above initial condition, let us denote the mapping from the
current process $Y$ to the temperature process $\Theta$ as $\Theta =
\xi_{\tau}(Y),$ where we emphasize the dependence on the thermal time
constants $\tau = (\tau_l)_{l \in \mathcal{L}}.$
\subsection{Stochastic model for power injections}
We now describe our stochastic model for the power injections $X(t)$.
Recall that in order to characterize the capacity region of the power
grid, we have to estimate the following overload
probabilities: $$\mathbb{P}(\norm{Y} \ge 1), \quad
\mathbb{P}(\norm{\Theta} \ge 1).$$ We use the theory of large
deviations to estimate these probabilities.
  Formally, we model the vector of random power injections
  $X^{\epsilon}(t)=(X_1^{\epsilon}(t),\ldots,X_m^{\epsilon}(t))$ as
  the strong solution of the $m$-dimensional stochastic differential
  equation (SDE)
  \begin{equation}
    \label{eq: SDE}
    dX^\epsilon(t)=b(X^\epsilon(t))dt+\sqrt{\epsilon}L(X^\epsilon(t))dW(t),
  \end{equation}
  where $X^{\epsilon}(0)=\mu,$ $b(x)=(b_1(x_1),\ldots,b_m(x_m))$,
  $L(x)=\mbox{diag}(\{l_i(x_i)\}_{i=1,\ldots,m})$ and
  $W(t)=(W_i(t))_{i=1,\ldots m}$.
  The function $b$ is referred to as
  the \emph{drift function}, and captures the evolution of the process
  in the absence of noise. The noise in the evolution of the process
  is introduced by the second term in \eqref{eq: SDE}: $W_i(t)$ is a
  standard Brownian motion in $\mathbb{R}$. This noise is modulated in
 a state-dependent fashion by the \emph{diffusion function} $L$, and the scaling parameter $\epsilon > 0$ captures the amount of randomness
  in the power injections. As $\epsilon \ra 0,$ the magnitude of the
  noise injected into the evolution of the process
  $X^{\epsilon}(t)$ diminishes, making large deviations from the
  ``noise-less" behavior exponentially (in $1/\epsilon$) unlikely.
  
   It is in this regime that LD theory gives us tractable
   approximations of the probabilities of the rare events
   corresponding to current and temperature overloads.  In
   practice,~$\epsilon$ can be chosen so that the variance of the
   process $X^{\epsilon}(t)$ matches the estimation error for
   renewable production over a specific time unit (Section~\ref{ss:
     118}).  We make the following regularity assumptions:$\forall
   i=1,\ldots, m,$ $b_i:\mathbb{R}\to\mathbb{R}$ is Lipschitz
   continuous and differentiable with $b_i(\mu_i)=0$;
   $l_i:\mathbb{R}\to(0,\infty)$ is Lipschitz continuous, bounded and differentiable.  

The $\epsilon-$scaled current process $Y^{\epsilon}(t) =
 (Y^{\epsilon})_{\ell \in \mathcal{L}}$ is defined as per the DC power
 flow equations: $Y^{\epsilon}(t) = C X^{\epsilon}(t) + y.$ Similarly,
 the $\epsilon-$scaled temperature process  $\Theta_{\tau}^{\epsilon}(t) =
 (\Theta_{\tau}^{\epsilon})_{\ell \in \mathcal{L}}$, with thermal constant $\tau$, is defined as
 $\Theta_{\tau}^{\eps} = \xi_{\tau}(Y^{\eps}).$ In the following sections, we
 apply the theory of large deviations to estimate the
 probabilities $\mathbb{P}(\norm{Y^{\eps}} \ge 1), \quad
\mathbb{P}(\norm{\Theta^{\eps}} \ge 1),$ in the limit as $\eps \da 0.$

\section{Capacity regions characterization based on current
  overload}\label{sc: current}
The traditional approach for ensuring line reliability is to impose
the condition
$\norm{Y(t)}:=\max_{\ell\in\mathcal{L}}|Y_{\ell}(t)| \le 1$ at
all times. In this spirit, in this section we characterize the
capacity region of the power grid obtained by bounding the probability
of current overload over $[0,T]$ by a prescribed target $q$
\[\mathbb{P}(\|Y\|_{\infty}\ge 1)\le q.\]
Our focus is to characterize the space of initial power
injections that can be `set' at time $0$, such that the probability
that the inherent variability in the stochastic sources leads to a
current overload before the next control instant is
small.\footnote{Given the equivalence between line currents and power
  flows under the DC approximation, the results in this section can
  also be interpreted in terms of the probability of exceeding line
  power flow limits.}

The above approach is in line with the conventional technique of
enforcing the thermal limits of transmission lines by capping the
peak current on each line. In Section~\ref{sc:temperature} a more refined approach, taking into account the transient relationship between line current and
line temperature, is presented.

In the following, we first provide a large deviations principle for the current overflow
event~$\{\|Y^{\epsilon}\|\ge1\}$ in the limit as $\epsilon \da 0.$
Next, we use this characterization to define the current-overload
based capacity region, and prove a convexity
result which facilitates its application as
a constraint in OPF. We then provide two lemmas that are useful for
computing the capacity region in practice and we give a
closed-form characterization of the capacity region when the stochastic
injections follow an Ornstein-Uhlenbeck process.
\subsection{Large deviations results}
The theory of large deviations (LD) is concerned with calculating the exponential decay of rare events probabilities, by means of the so-called \textit{rate functions}.
The main idea behind the theory is to provide a rigorous mathematical foundation to the approximation
\[
\mathbb{P}_{\epsilon}(E) = \int_{x\in E} f_{\epsilon} (x) dx \approx \max_{x\in E} f_{\epsilon} (x)
\]
in the regime where $\epsilon$ is small.
It turns out that the density $f_{\epsilon} (x)$ (where $x$ can be a function, or \textit{path}, on a interval $[0,T]$) can be approximated further, as is often done in the asymptotic analysis of integrals. More precisely, a family of probability measures $\mathbb{P}_{\epsilon}$ on a polish space $\mathcal{X}$ is said to satisfy a \textit{large deviation principle}~\cite{Dembo1998} with \textit{rate function} $I$ if, for all Borel measurable set $E\subset {\mathcal {X}}$,
\begin{align}\label{eq: LDP-general}
-\inf _{{x\in E^{\circ }}}I(x)&\leq \liminf _{\epsilon\to 0}\epsilon\log {\big (}{\mathbb  {P}}_{\epsilon}(E){\big )}\\&\leq
\limsup _{\epsilon\to 0}\epsilon\log {\big (}{\mathbb  {P}}_{\epsilon}(E){\big )}\leq -\inf _{{x\in {\bar  {E}}}}I(x).
\end{align}

The reason to work with a $\liminf$ and $\limsup$
is mainly technical, and often we can interpret Eq.~\eqref{eq: LDP-general} simply as $\mathbb{P}_{\epsilon}(E) \approx \exp\{-\inf _{x\in { {E}}}I(x)/\epsilon\}$.
That is, every path $x$ towards the rare event $E$ has a certain ``cost" $I(x)$,
and the LDP tells us  that, as $\eps \to 0$, the path with the smallest cost yields the largest probability and is therefore the most likely way for the event to happen.
For further background, and other engineering applications of LD we refer to~\cite{Bucklew1990}; for an introduction of large-deviations theory aimed at physicists, see~\cite{Touchette2009}.

This subsection is based on the
Freidlin-Wentzell theory (F-W,~\cite{Dembo1998}), which is concerned with large deviation principles for the \textit{paths} of a stochastic process.
Thanks to Theorem~$5.6.7$ in~\cite{Dembo1998}, the power injections process $X^\epsilon$ satisfies a sample
path large deviations principle (SPLDP) over $C_{\mu}([0,T])=\{g:
[0,T]\to \mathbb{R}^m : g \text{ is continuous and }g(0)=\mu\}$, with
good rate function
 \begin{equation}\label{eq: rate power}
   \IP(g)=
  \sum_{i=1}^m \IPi(g_i).
 \end{equation}
 Here, $g=(g_1,\ldots.g_m)$ and $\IPi$ is the good rate function for
 the SPLDP associated with the process $X_i^\epsilon(t)$,
 $i=1,\ldots,m$, and it is given by
 \[ \IPi(g_i)=\begin{cases}
   \frac{1}{2}\int_{0}^T \Bigl(\frac{g\pr_i-b_i(g_i)}{l_i(g_i)}\Bigr)^2dt&\mbox{if }g_i\in H^1_{\mu_i}(\mathbb{R}),\\
   \qquad\qquad\infty &\mbox{if }g_i\notin H^1_{\mu_i}(\mathbb{R}).\\
 \end{cases}
 \]
 Here, $H^1_{\mu}(\mathbb{R}^m):=\{g:[0,T]\to\mathbb{R}^m\,:\,g(t)=\mu
 +\int_{0}^t\phi(s)ds, \phi\in L_2([0,T])\}$ is the space of
 absolutely continuous functions with value $\mu$ at $0$ and which
 possess a square integrable derivative.
Next, we apply a very useful tool in LD theory, known as the
  \textit{Contraction Principle}, which allows to map large deviations
  principles from one space to another.  Thanks to the Contraction
Principle, Theorem $4.2.1$ in~\cite{Dembo1998} and Eq~\eqref{eq: DC
  norm current exp}, the current process $Y^\epsilon$ satisfies a
SPLDP with good rate function
\[\IC(f)=\inf_{\substack{g\in H^1_{\mu}:\\ y+Cg=f}}\IP(g).\]
Thanks to Lemma~\ref{lm: rank C}, the matrix $C$ has linear
independent columns. Therefore, its Moore-Penrose inverse has an
explicit formula $C^+=(C^{\top}C)^{-1}C^{\top}$ and it is a left inverse of
$C$. Thus, for $f\in y+ C(H^1_{\mu}(\R^m))\subset
H^1_{\nu}(\R^L)$ the equation $y+Cg=f$ has unique solution
$g=C^+(f-y)$, yielding
{\small
\begin{equation}\label{eq: rate current}
 \IC(f)=\begin{cases}
 \IP(C^+(f-y))\qquad &\text{if } f\in y+ C(H^1_{\mu}(\R^m)),\\
 \infty & \text{otherwise}.
 \end{cases}
\end{equation}
}
For the current overload event we then have that
\begin{align}
 \limsup_{\epsilon\to0}\epsilon\log \mathbb{P}(\|Y^{\epsilon}\|\ge1)= -\IC^*, \label{eq: LDP current}\\
  \IC^*=\inf_{\substack{f\in y+CH^1_{\mu}:\\ \|f\|\ge1}} \IC(f)=
\inf_{\substack{g\in H^1_{\mu}:\\ \|y+Cg\|\ge1}} \IP(g) \label{eq: rate* current},
\end{align}

with $\IC^*$ the \textit{decay rate} for the current overload event. \footnote{
Note that $f\in H^1_{\nu}\setminus (y+CH^1_{\mu})$ implies $\IC(f)=\infty$, thus $\IC^*=\inf_{\substack{f\in y+CH^1_{\mu}:\,
      \|f\|\ge1}} \IC(f)=\inf_{\substack{f\in H^1_{\nu}:\, \|f\|\ge1}} \IC(f)$}
\subsection{Capacity region based on current overload}\label{sc:
  capacity current}
Eq.~\eqref{eq: LDP current} yields the following approximation
for the current overload probability for small $\epsilon$:
\begin{equation}\label{eq: current approx}
\mathbb{P}(\|Y^{\epsilon}\|_{\infty}\ge1)\approx e^{-\IC^*(\overline{\mu})/\epsilon}.
\end{equation}
We use the above approximation to define the capacity region
for the power grid, based on the constraint that the probability of
current overflow must not exceed $p,$ where $p > 0$ is a small
pre-defined threshold:
\begin{equation}\label{eq: capacity current}
  \Rcur:=
  \{\overline{\mu}\in\mathbb{R}^N\,:\,\IC^*(\overline{\mu}) \geq -\eps\log(p)\}.
\end{equation}
In the remainder of this section, we shed light on structural
properties and computational aspects of this capacity region. Our
first result shows that the capacity region is convex with respect to
the deterministic power injections.

 \begin{lemma}\label{lm: convex current}
   $\Rcur$ is convex in the deterministic
   power injections vector $\mu_D$.
 \end{lemma}

Lemma~\ref{lm: convex current} is important as convexity enables the
set of allowable deterministic injections to be incorporated as a
constraint in OPF problems (see, for example,
\cite{Bienstock2014}). For the special case where power injections are
modeled as an Ornstein Uhlenbeck process, we show in Section~\ref{ss:
  OU current} that the capacity region
$\Rcur$ itself is convex. Letting
\[\psi_{\ell}=\inf_{\substack{f\in H^1_{\nu}:\\
    \|f_{\ell}\|_{\infty}\ge1}} \IC(f)=\inf_{\substack{g\in
    H^1_{\mu}:\\ \|y_{\ell}+C_{\ell} g\|_{\infty}\ge1}} \IP(g),\]
with $C_{\ell}$ being the the $\ell$-th row of matrix $C$, we note that
$\IC^*=\min_{\ell \in\mathcal{L\pr}}\psi_{\ell}$, where $\mathcal{L}\pr:=\{\ell \in\mathcal{L}\,:\, C_{\ell}\ne0\}$.
\footnote{ote that if $C_{\ell}=0$, then
$Y_\ell^{\epsilon}(t)=Y_\ell^{\epsilon}(0)=y_{\ell}$ is constant and
  $|y_{\ell}|=|\nu_{\ell}|<1$, yielding $\displaystyle\psi_{\ell}=\inf_{\substack{g\in H^1_{\mu}:\,\|y_{\ell}\|_{\infty}\ge1}} \IP(g)=\infty$.}.

In other words, the
decay rate for a current overload in the network is the minimum of the decay
rates corresponding to the overload of each line. Decay rates, together with Eq.\ref{eq: current approx}, provide an analytical tool to rank transmission lines in terms of their vulnerability~\cite{Nesti2018}.
The next lemma shows that the current overload on any link most likely
occurs at the end time.
\begin{lemma}
   \label{lm: overflow_boundary}
$\forall \,\ell \in \mathcal{L\pr},$
   $\displaystyle\psi_{\ell} = \inf_{\substack{g\in H^1_{\mu}:\,|y_{\ell}+C_{\ell} g(T)|=1}} \IP(g).$
 \end{lemma}
 For
$a \neq \nu_{\ell},$ define
  \begin{equation}\label{eq:variational_prob}
  \psi_{\ell}^{(a)} = \inf_{\substack{f\in y+CH^1_{\mu}:\, f_{\ell}(T)=a}} \IP(f),
  \end{equation}so that $\psi_{\ell}
  =\psi_{\ell}^{(1)}\wedge\psi_{\ell}^{(-1)}=\min{\psi_{\ell}^{(1)},\psi_{\ell}^{(-1)}}$ and
  \begin{equation}
    \label{eq:curr_dr_1}
    \IC^*=\min_{\ell
      \in\mathcal{L}\pr}\psi_{\ell}^{(1)}\wedge\psi_{\ell}^{(-1)}.
  \end{equation}
Equation~\eqref{eq:curr_dr_1} allows us to rewrite the capacity
region as
{\small
    \begin{equation}
    \label{eq: current_rewritten}
    \Rcur=
    \bigcap_{\substack{\ell \in\mathcal{L}\pr,\,a \in
        \{-1,1\}}}\{\overline{\mu}\in\mathbb{R}^N\,:\,\psi_{\ell}^{(a)}
    \geq -\eps\log(p)\}.
    \end{equation}
   }
    Thus, obtaining the capacity region
    $\Rcur$ hinges on computing
    $\psi_{\ell}^{(a)},$ which by definition is the solution of \eqref{eq:variational_prob}. To solve this variational problem with boundary constraints, one can for instance use the Euler - Lagrange equation (see also our discussion in Section~\ref{sc: numerics}).
For simple diffusion models, this approach can be used to obtain the
optimal path and $\psi_{\ell}^{(a)}$ in closed form, leading to an
explicit characterization of the capacity region
$\Rcur.$ Next, we illustrate this for the case where the power injections are modeled as a OU process.
\subsection{Explicit computations for Ornstein-Uhlenbeck process}\label{ss: OU current}
In this section we suppose that the power injections $X^{\epsilon}(t)$
follow a multivariate Ornstein-Uhlenbeck (OU) process, which
  is the most tractable example of an SDE and is, in particular, a
  Gaussian process.\footnote{The gaussianity assumption for
      wind power is debatable. While consistent with atmospheric
      physics~\cite{Bienstock2014} and recent wind park
      statistics~\cite{Kolumban2017,Berg2016}, different models are
      preferred for different
      timescales~\cite{MilanWaechterPeinke2013}.}
Such a process is of the form
 \begin{equation}\label{eq: OU current expl}
    dX^\epsilon(t)=D(\mu-X^\epsilon(t))dt+\sqrt{\epsilon}LdW(t)
  \end{equation}
i.e. the functions
$b(\cdot)$ and $L(\cdot)$ in the SDE~\eqref{eq: SDE} are $
b(x)=D(\mu-x)$ and $L(x)=L$,
where $D=\diag(\{\gamma_i\}), L=\diag(
\{l_i\})$, and $\gamma_i,l_i>0$ for all
$i=1,\ldots,m.$ For this model, the capacity region can be expressed
in closed form, as shown in the next Proposition.\footnote{Note that our framework allows to extend Proposition~\ref{pr: OU_regions} to mixtures of OUs, providing flexibility to the modeler while keeping the benefits of closed-form expressions.}
\begin{proposition}\label{pr: OU_regions}
  If $X^{\epsilon}(t)$ is defined by~\eqref{eq: OU current expl}, then
 \begin{equation*}\label{eq: capacity current OU}
 \begin{split}
   \Rcur=\bigcap_{{\ell}\in\mathcal{L\pr}}
   \Bigl\{\overline{\mu}\in\mathbb{R}^N\,:\,|\nu_{\ell}| \leq 1-\sqrt{\eps\log(1/p)C_{\ell}M_TC_{\ell}^{\top}}\Bigr\}.
 \end{split}
 \end{equation*}
  In the particular case $D=\gamma I$, Eq.~\eqref{eq: capacity current OU} simplifies to
  \begin{equation}\label{eq: capacity current OU_simplified}
 \begin{split}
    \Rcur=\bigcap_{{\ell}\in\mathcal{L}\pr}
 \Bigl\{\overline{\mu}\in\mathbb{R}^N\,:\,|\nu_{\ell}| \leq 1-\beta_{\ell}\Bigr\}
 \end{split}.
 \end{equation}
Here, $M_t=L^2D^{-1}(I-e^{-2Dt})e^{D(t-T)}$ and
\[\beta_{\ell}:=\sqrt{\frac{(1-e^{-2\gamma T})\eps\log(1/p)\sigma_{\ell}^2}{\gamma}},\quad \sigma_{\ell}^2:=C_{\ell}L^2C_{\ell}^{\top}.\]
\end{proposition}
We make the following remarks regarding Proposition~\ref{pr: OU_regions}: (i)  $\Rcur$ is a closed convex
set; in particular, it is a polyhedron in $\mathbb{R}^N$. We note that
this property enables us to incorporate the capacity region in OPF problems.
(ii) $\beta_{\ell}$ is a strictly
  decreasing function of $\gamma,$ implying that $\Rcur$ shrinks as $\gamma$ becomes smaller. This is intuitive, since for
  small values of $\gamma,$ the OU process will revert to its
  long-term mean $\mu$ with less force;
  (iii)
the longer the time $T$ between two control instants, the greater the probability that the  fluctuations in the power injections will result in an overload, yielding a smaller $\Rcur$;
 (iv) the expression for $\Rcur$
encloses in a single formula the dependency on the initial condition $\nu$, on the window length $T$, and on the topology of the
  network, the physical properties of the transmission lines and the
  evolution of the stochastic power injections, encoded in the matrices $C,L,D$.
\section{Capacity regions characterization based on temperature
  overload}\label{sc: temperature}
%


Since temperature responds gradually to current, a current overload of
a short duration does not necessarily imply an overload in
temperature.
By explicitly capturing the
transient relationship between temperature and current, we can enlarge the conservative capacity region obtained in Section~\ref{sc: current}.
In the following, we first provide a large deviations principle for the temperature overload event
$\mathbb{P}(\|\Theta^{\epsilon,\tau}\|_{\infty}\ge1).$ Then, we define the temperature-overload based capacity region and prove a convexity result for it, analogous to
the result in Section~\ref{sc: current}.

However, due to the non-local in time relationship between current and temperature, the decay rate for the temperature process is hard to compute explicitly. As a result, the capacity region cannot be expressed in closed form for even the simplest diffusion models. To address this issue,
we develop two approximations of the capacity region: the first is an
inner bound, while the second is based on a first order Taylor expansion of the decay rate around $\tau = 0.$ These approximations have the following appealing
properties, which make them amenable to application in OPF
formulations. Firstly, both approximations are supersets of
the current-based capacity region $\Rcur$. Secondly, they have the same computational
complexity as $\Rcur.$ Thirdly, for the
special case where the stochastic power injections are modeled by an
OU process, both regions can be expressed in closed
form (Subsections~\ref{ss:LB OU},~\ref{ss: OU taylor}). Finally, both approximations are convex over the deterministic power injections.
\subsection{Capacity region based on temperature overload}\label{sc:temperature}
Thanks to the relationship~\eqref{eq: ode}, the contraction principle yields that
$\Theta_{\epsilon,\tau}$ satisfies a SPLDP with good rate function
\begin{equation}\label{eq: rate temperature}
 \IT(h)=\inf_{\substack{f\in H^1_{\nu}:\\ \xi_{\tau}(f)=h }} \IC(f)=
 \inf_{\substack{f\in y+CH^1_{\mu}:\\ \xi_{\tau}(f)=h }} \IC(f).
\end{equation}
For the temperature overload event we thus have
 \begin{align}
 & \limsup_{\epsilon\to0}\epsilon\log \mathbb{P}(\|\Theta^{\epsilon,\tau}\|_{\infty}\ge1) \le -\IT^*, \label{eq: LDP temperature}\\
  &\IT^*=\inf_{\substack{ h\in\xi_{\tau}(H^1_{\nu})\\ \|h\|\ge1}} \IT(h), \label{eq: rate* temperature}
\end{align}
where $\IT^*$ is the \textit{temperature decay rate}. Letting, for $\ell\in \mathcal{L}\pr$,
 \begin{equation}\label{eq: variational problem temperature}
 \omega_{\ell} = \inf_{\substack{h\in\xi_{\tau}(H^1_{\nu}):\\ \norm{h_{\ell}} \geq 1}} \IT(h)=
 \inf_{\substack{g\in H^1_{\mu}:\\ \norm{\xi_{\tau_\ell}(y_{\ell}
       +C_{\ell}g)} \geq 1}} \IP(g),
       \end{equation}
 we see that the decay rate for the temperature is
$\IT^* = \min_{{\ell} \in \mathcal{L}\pr} \omega_{\ell}.$ Note that $\omega_{\ell}$ and $\IT^*$ depend on
  $\overline{\mu},\tau$ and $T$.  As before, Eq.~\eqref{eq: LDP
    temperature} yields the following approximation for the rare event
  probability, for small $\epsilon$:
\begin{equation}\label{eq: temperature approx}
\mathbb{P}(\|\Theta^{\epsilon,\tau}\|_{\infty}\ge1)\approx e^{-\IT^*(\overline{\mu})/\epsilon}.
\end{equation}
This leads to the following definition of the capacity region
 \begin{equation}\label{eq: capacity temperature}
 \begin{split}
    \Rtemp&:=
      \{\overline{\mu}\in\mathbb{R}^N\,:\,\IT^*(\overline{\mu})\ge-\eps\log(p)\}\\
  &= \bigcap_{l\in\mathcal{L}\pr}\{\overline{\mu}\in\mathbb{R}^N\,:\,\omega_{\ell}(\overline{\mu})\ge-\eps\log(p)\}.
  \end{split}
 \end{equation}
 We have the following convexity result:
\begin{lemma}\label{lm: convex temperature}
 $\Rtemp$ is convex in the deterministic power injections vector $\mu_D$.
 \end{lemma}

The variational problem for the temperature overload~\eqref{eq: rate*
  temperature} is difficult to solve in general, and numerics can also
prove to be challenging.
Motivated by this difficulty, in the next subsection we develop approximations for the temperature decay rate, and the corresponding capacity regions, by reducing ~\eqref{eq: rate*  temperature} to the easier problem~\eqref{eq:variational_prob}.
\subsection{Inner bound for the capacity region}
\label{sc:temp_inner_bound}
In this section, we develop an inner bound for the capacity region
$\Rtemp$ which is larger
than the capacity region $\Rcur$ based on
current overload, and thus captures some of the benefit of
incorporating temperature dynamics. Define
\[\IT^{(LB)}:=\min_{{\ell}\in\mathcal{L}\pr}\psi_{\ell}^{(\alpha_{\ell})}\wedge
\psi_{\ell}^{(-\alpha_{\ell})}.\] The next lemma shows that $\IT^{(LB)}$ is a lower bound for the temperature decay rate, i.e. $\IT^*\ge\IT^{(LB)}.$
\begin{lemma}
  \label{lm: temp_bound}
For all $\ell\in\mathcal{L\pr}$, we have
$\omega_{\ell} \geq \psi_{\ell}^{(\alpha_{\ell})}\wedge\psi_{\ell}^{(-\alpha_{\ell})},$
where $\alpha_{\ell} = \sqrt{\frac{1-\nu_{\ell}^2 e^{-T/\tau_{\ell}}}{1-e^{-T/\tau_{\ell}}}}.$
\end{lemma}
The capacity region based on the lower bound $\IT^{(LB)}$ is
\begin{equation*}
 \begin{split}
   \Rlb :=
   \{\overline{\mu}\in\mathbb{R}^N\,:\,\IT^{(LB)}(\tau,\overline{\mu}) \geq -\eps\log(p)\}
  \end{split}
 \end{equation*}
 The following proposition states that the capacity region based on the lower bound, while being an inner approximation of the actual temperature-based region, is
 less conservative than the current-based capacity constraint.
\begin{proposition}\label{pr: LB}
  $\IC^*\le\IT^{(LB)}\le\IT^*$ and
  $\Rcur\subseteq\Rlb
  \subseteq \Rtemp$ for all
  $\tau\ge0$.
\end{proposition}
As a consequence, using
 $\Rlb$ over
 $\Rcur$ allows for larger power
 injections values (i.e., less curtailment), while still bounding the probability of a temperature overload and without additional computational burden.
 Finally, we note that the inner bound satisfies the
 following convexity property.
 \begin{lemma}\label{lm: convex lb}
   $\Rlb$ is convex in the
   deterministic power injections vector $\mu_D$.
 \end{lemma}
 The proof goes along the same lines of the proofs of Lemmas
 \ref{lm: convex current} and \ref{lm: convex temperature} and is therefore omitted.
 \subsection{Taylor approximation of the decay rate and corresponding capacity regions}
 \label{taylor_general}
 In this section we derive a heuristic approximation for the temperature decay rate
\begin{equation}
\IT^*=
 \inf_{\substack{h\in \xi_{\tau}(y+CH^1_{\mu}),\,\|h\|\ge 1 }} \IT(h),
\end{equation}
based on a Taylor expansion around $\tau=0$. First, write the temperature rate function~\eqref{eq: rate temperature} as
{\small \begin{align*}
&\IT(h)= \begin{cases}
G(\tau,h)\quad &\text{if } h\in \xi_{\tau}(y+C H_{\mu}^1),\\
\infty &\text{otherwise},
\end{cases}
 \end{align*}}
where $G(\tau,h)$ is defined explicitly in the Appendix.

\begin{Taylor approximation}\label{th: taylor}
Let $f_*$ be the optimal current path to overflow. For small $\tau$, we will use the approximation
\begin{equation}\label{eq: taylor first order}
 \IT^*\approx\IT^{(TL)}:=\IC^*+\tau \cdot\nabla_{\tau}G(\tau,f_*^2)|_{\tau=0},
\end{equation}
where $G$ is defined in
\noindent If $\tau$ is of the form $\tau=\tau_0(1,\ldots,1)^{\top}$, $\tau_0>0$, we obtain the closed-form expression
\begin{equation}
\label{eq:closed_form}\IT^{(TL)}=\IC^*+\tau_0\Phi_{f_*}\end{equation}
where
{\small\begin{align}\label{eq:calculationstaylor}&\Phi_{f_*}:=\sum_{i=1}^m\Bigl[ K_i(f_*(T),f_*\pr(T))-K_i(f_*(0),f_*\pr(0))\Bigr],\\
&K_i(f_*(t),f_*\pr(t)):=
	      \frac{1}{2} \Bigl(\frac{C_i^+f_*\pr(t)-b_i(C_i^+(f_*(t)-y))}{l_i(C_i^+(f_*(t)-y))}\Bigr)^2.
\end{align}}
In particular the approximation $\IT^{(TL)}$ depends only on the current decay rate $\IC^*$ and on the values
$f_*(0),(f_*)\pr(0),f_*(T),(f_*)\pr(T)$.
\end{Taylor approximation}

 The heuristic is motivated by the formal Taylor expansion of
 $I_{t,\tau}^*$ around $\tau=0$,i.e. $\mathcal{I}_{t,0}^*+\tau \cdot\nabla_{\tau}I_{t,\tau}^*|_{\tau=0}+o(\tau).$
If $\tau=0$, the optimal temperature path to overflow is
$h_*=(f_*)^2$, so
$\IT^*=\ICS(h_*)=
 \IC(f_*)=\IC^*
,$
and the substitution of $\nabla_{\tau}G(\tau,f_*^2)|_{\tau=0}$ for
$\nabla_{\tau}\IT^*|_{\tau=0}$ is motivated by an
infinite-dimensional version of Danskin`s Theorem~\cite{Bonnans2000}, Proposition $4.13$. To make this rigorous is quite challenging, as the feasible sets in our variational problem depend in a rather intricate way on $\tau$. The explicit calculations for the case $\tau=\tau_0(1,\ldots,1)^{\top}$, $\tau_0>0$, are reported in the appendix.

Eq.~\ref{eq:closed_form} provides an approximation of the temperature decay rate which depends only on
the current decay rate and the corresponding optimal path, which are generally easier to obtain.
The capacity region corresponding to the Taylor approximation is
{\small\begin{equation*}\label{eq: temp capacity regions general}
\begin{split}
    \Rtaylor:=
  \{\overline{\mu}\in\mathbb{R}^N\,:\,\IC^*(\bar{\mu})+\tau_0\Phi_{f_*}\ge-\eps\log(p)\}.
\end{split}
  \end{equation*}
}
In section
\ref{ss: OU taylor} we will see that in the OU case the inequality $\IT^{(TL)}\ge\IC^*$ holds, and thus $\mathcal{\tilde{R}}^{(t,\tau,TL)}_{\eps,p}\supseteq
  \Rcur$,
confirming the intuition that the temperature-based approach is less conservative than the current-based one.

 \subsection{Explicit computations for OU: lower bound}\label{ss:LB OU}
  In this section we assume that the power injection process $X^{\epsilon}(t)$ follows the OU process~\eqref{eq: OU current expl}, and we explicitly compute the lower bound $\IT^{(LB)}$ and the corresponding capacity region $\Rlb$.
\begingroup
\allowdisplaybreaks
   \begin{proposition}\label{pr: delta,eta}
  If $X^{\eps}(t)$ is defined by~\eqref{eq: OU current expl}, then
 \begin{align*}
&\IT^{(LB)}=\min_{\ell\in\mathcal{L}\pr}\frac{(\alpha_{\ell}-|\nu_{\ell}|)^2}{C_{\ell}M_TC_{\ell}^{\top}},\\
   & \Rlb=
  \bigcap_{\ell \in\mathcal{L\pr}}\{\overline{\mu}\in\mathbb{R}^N:
 \frac{(\alpha_{\ell}-|\nu_{\ell}|)^2}{C_{\ell}M_TC_{\ell}^{\top}}\ge-\eps\log(p)\},
 \end{align*}
where {\small\[\alpha_{\ell} = \sqrt{\frac{1-\nu_{\ell}^2 e^{-T/\tau_{\ell}}}{1-e^{-T/\tau_{\ell}}}},\quad
M_t=L^2D^{-1}(I-e^{-2Dt})e^{D(t-T)}.\]}
 In the particular case $D=\gamma I$, we have
 \begin{equation}\label{eq: capacity LB OU unif}
 \begin{split}
    \Rlb:=
  \bigcap_{\ell \in\mathcal{L\pr}}&\{\overline{\mu}\in\mathbb{R}^N\,:\,
|\nu_{\ell}|\le\delta_{\ell}\}
   \end{split},
 \end{equation}
 {\small
 \begin{align*}\label{eq: delta,eta}
 &\delta_{\ell}=
\sqrt{1-\eta_{\ell}^2e^{-T/\tau_{\ell}}(1-e^{-T/\tau_{\ell}})} - \eta_{\ell}
(1-e^{-T/\tau_{\ell}}),\\
& \eta_{\ell} :=
\sqrt{\frac{\epsilon \log(1/p)\sigma^2_{\ell}(1-e^{-2\gamma T})}{\gamma}} <
1,\,\,\sigma_{\ell}^2=C_{\ell}L^2C_{\ell}^{\top}.
\end{align*}
}
  \end{proposition}
\endgroup
If $D=\gamma I$, we see from Prop.~\ref{pr: delta,eta} that $\Rlb$ is a convex polyhedron in $\mathbb{R}^N$, as in the case of the current region, and is in particular a scaled version of the polyhedron $\Rcur$. Moreover, $\delta_{\ell}\in(1-\eta_{\ell},1),
\delta_{\ell}\xrightarrow[]{\tau\to\infty} 1$ and $\delta_{\ell}\xrightarrow[]{\tau\to 0} \eta_{\ell}$. This means that, as  $\tau$ increases, the capacity region~\eqref{eq: capacity LB OU unif}
gets closer to the larger region $\{\overline{\mu}\in\mathbb{R}^N\,:\,\|\nu\|_{\infty}<1\}$,
which is the stability region for a deterministic system. On the other hand, as $\tau\to 0$, the region in~\eqref{eq: capacity LB OU unif} boils down to the smaller current-based capacity region given in~\eqref{eq: capacity current OU}.
    \subsection{Explicit computations for OU: Taylor approximation}\label{ss: OU taylor}
    In this section we consider again the OU process $X^{(\epsilon)}$~\eqref{eq: OU current expl} in the particular case $D=\gamma I$, and we develop
      the capacity regions based on the Taylor approximation~\ref{th: taylor}.

    \begin{proposition}\label{pr: OU_taylor}
    For  $\tau=\tau_0(1,\ldots,1)^{\top}$ we have
{\small
\begin{equation*}\label{eq: taylor first order OU}
\begin{split}
\IT^{(TL)}=(1+2\tau_0\gamma)\,\IC^*(\bar{\mu})=
 (1+2\tau_0\gamma)\min_{\ell\in\mathcal{L\pr}}\frac{(1-|\nu_{\ell}|)^2}{C_{\ell}M_TC_{\ell}^{\top}}
 ,
\end{split}
\end{equation*}
\begin{equation*}
\begin{split}
    \Rtaylor
    =&\bigcap_{{\ell}\in\mathcal{L\pr}}\Bigl\{\overline{\mu}\in\mathbb{R}^N\,:\,|\nu_{\ell}|\le 1-\eta_{\ell}/\sqrt{1+2\tau_0\gamma}\Bigr\},
\end{split}
\end{equation*}
}
   \end{proposition}

It is clear that $\mathcal{\tilde{R}}^{(t,t,\tau_0L)}_{\eps,p}$ is a convex polyhedron, as it was the case for the current region $\Rcur$ and the
lower bound region $\mathcal{\tilde{R}}^{(t,\tau_0,LB)}_{\eps,p}$.
Moreover, since $1+2\tau_0\gamma>0$, we see that $\Rtaylor\supseteq \Rcur$ and in particular
$\Rtaylor$ is a re-scaled version of $\Rcur$. Recall that this was also the case for the lower bound capacity region: the difference is that, while the lower bound holds for every $\tau>0$, the approximation $\IT^{(TL)}$
is good only for small $\tau_0$.
In general,
$\Rtaylor$ and $\Rlb$ are not subsets of each other.

\section{Numerics}\label{sc: numerics}

In order to compute the temperature decay rate $\IT^*$, one has to solve the variational problem in Eq.~\eqref{eq: variational problem temperature}, which is computationally harder than the one for the current in Eq.~\eqref{eq:variational_prob}, due to the integral mapping in Eq.~\eqref{eq: xi}.

  On the other hand, the theory we presented enables us to reduce the computation of the decay rates
$\IT^{(LB)},\IT^{(TL)}$ to the easier variational problem for $\IC^*,$ capturing the benefits of incorporating the temporal dynamics between current and temperature without additional cost. Variational problems like~\ref{eq:variational_prob}, which are based on F-W theory, are well studied in the literature, and when closed-form expression are not available efficient numerical algorithms have been developed~\cite{Heymann2008}.

In the next subsections we apply our theory to derive the capacity
regions for two IEEE test cases, and we quantify the capacity gains
achieved by $\Rlb,\Rtaylor$ over $\Rcur$ assuming an OU model for
power injections. The code used to produce Figures~\ref{fig:14bus apart},\ref{fig:118bus} is available at \href{https://github.com/TommasoNesti/Temperature-Overloads}{\textit{github.com/TommasoNesti/Temperature-Overloads}}.

\paragraph*{Scalability} Thanks to the analytic characterization of capacity regions for the OU model, our approach is fully scalable and can effortlessly be applied to larger power networks, as there is virtually no computational burden in computing $\Rcur$ and, therefore, all the other capacity regions. For a detailed analysis on computational costs for solving~\ref{eq:variational_prob} for a general SDE, the interested reader is referred to~\cite{Heymann2008}, Section $3.3$.

 \subsection{IEEE-\texorpdfstring{$14$}{} test network}
In this section we develop capacity regions for the IEEE-14 test network, corresponding to the test case \textit{case14} in~\cite{Zimmerman2011}.

The grid consists of $14$ nodes and $20$ lines, and the original test case has constant deterministic power injections $P_{D}\in\mathbb{R}^{14}$, expressed in the \textit{per-unit} system.
We replace two of the deterministic injections (nodes $2$ and $13$) by OU processes with long term mean equal to the original deterministic power injection, and we assume we control the injections at nodes $6$ and $9$.
The test case reports the parameters $P_D$, $w_{ij}$ and $\widetilde{C}$, but does not include line limits, which we define as follows. For each line $\ell$ we set the maximum permissible current $I_{\ell,\max}=K|I_{\ell,\text{nom}}|$, $I_{\ell,\text{nom}}$ being the nominal current in line $\ell$ obtained from $P_{D}$ via the DC power flow equation, and $K=1.5$. We set $T=1,\gamma_i=1,l_i^2=10,\epsilon=0.25$ and $\tau=0.5$.

We compute two-dimensional capacity regions, which correspond to the amount of power that can be injected at the controllable sources so that the probability of overload in $[0,T]$ is sufficiently small.  The current-based capacity region is
\begin{align*}
&\Rcur=\{(\overline{\mu}_6,\overline{\mu}_9)\in\mathbb{R}^2\,|,\,
\overline{\mu}=(P_{D,2},\ldots,P_{D,5},\overline{\mu}_6,P_{D,7},\\
&\qquad\qquad P_{D,8},\overline{\mu}_9,P_{D,10},\ldots,P_{D,14}),\,\IC(\overline{\mu})\ge-\epsilon\log(p)\},
\end{align*}
and the other regions are defined similarly.

In Figures~\ref{fig:14bus aparta},~\ref{fig:14bus apartb} the $2$-dimensional capacity regions $\Rcur,\Rlb,\Rtaylor$ (denoted as $\tilde{\mathcal{R}}^{(\text{cur})},\tilde{\mathcal{R}}^{(\text{LB})},\tilde{\mathcal{R}}^{(\text{TL})}$ in the legend) are shown for two different
 target probabilities $p$, together with the region corresponding to a deterministic system
 \begin{equation}
 \Rdet=\{(\overline{\mu}_6,\overline{\mu}_9) \in\mathbb{R}^2\,|\, |\nu_{\ell}|\le 1 \quad
 \forall \ell=1,2,3\}
 \end{equation}

In particular, Figure \ref{fig:14bus apartb} shows that for $p=10^{-7}$ the lower bound region $\Rlb$ is more
  than two times bigger than $\Rcur$, and the Taylor region $\Rtaylor$
  is approximately two times bigger than $\Rlb$. This result suggests that for small target probabilities,
 the temperature-based approach yields a significative capacity gain.

Another application of the proposed methodology is the identification of the most vulnerable parts of the grid. For a given value of $\overline{\mu}$, let
$\ell^*(\overline{\mu}):=\text{argmin}_{\ell\in\mathcal{L}\pr}\psi_{\ell}^{(1)}(\overline{\mu})
  \wedge\psi_{\ell}^{(-1)}(\overline{\mu})$
   denote the line with the highest chance of overloading~\eqref{eq:curr_dr_1}, and, for a line $k\in\mathcal{L}\pr$, define
\begin{align}\label{eq:eq:partition}
 S_{k}:=\{(\overline{\mu}_6,\overline{\mu}_9)\in\mathbb{R}^2\,|\, \|\nu\|\le1, \ell^*(\overline{\mu})=k\}\subset \Rdet.\end{align}
 
 The region $S_{k}\subset \Rdet$ characterizes the \textit{controllable} power injections such that, in the event of large fluctuations of stochastic power injections, line $k$ is the most likely line to overload.
 The $S_k$-s partition $\Rdet$ in several subregions, as shown in Figure~\ref{fig:14bus apartc}.
Such characterization can help detecting the most vulnerable components of the grid: in this case, line $(12,13)$, corresponding to the biggest sub-region in Fig.~\ref{fig:14bus apartc}. Finally, Fig.~\ref{fig:14bus apartd} shows the topology of the network.

\begin{figure}[h!]
  \begin{subfigure}[t]{0.45\columnwidth}
    \includegraphics[width=0.95\columnwidth]
 {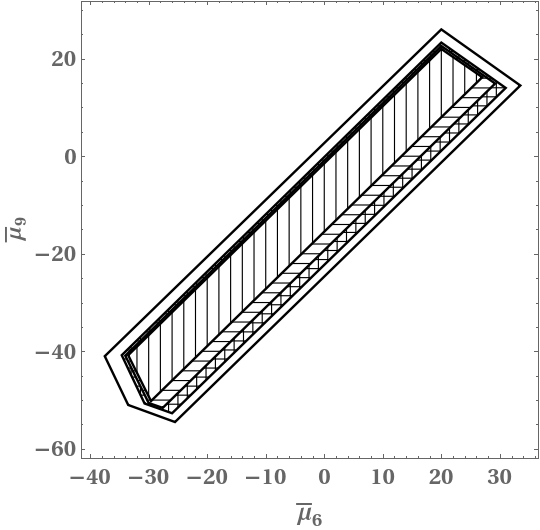}
     \caption{\footnotesize{Target probability $p=10^{-4}$.}}
   \label{fig:14bus aparta}
  \end{subfigure}
  \hspace{-1cm}
\begin{subfigure}[t]{0.55\columnwidth}
          \includegraphics[width=0.95\columnwidth]
 {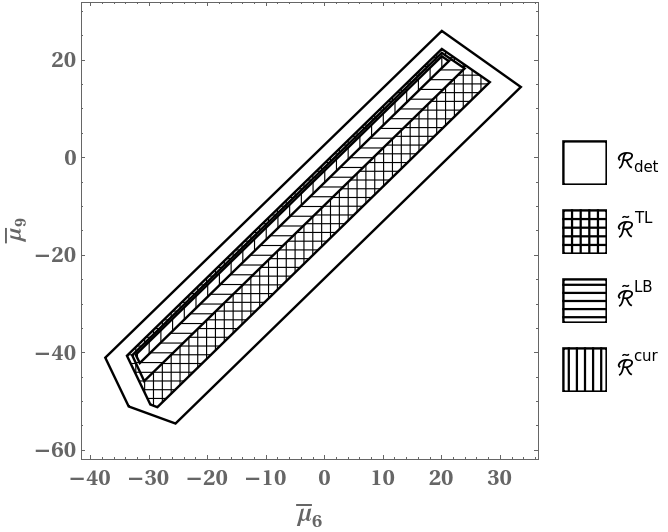}
          \caption{\footnotesize{Target probability $p=10^{-7}$.}}
\label{fig:14bus apartb}
  \end{subfigure}
   \begin{subfigure}[t]{0.55\columnwidth}
          \includegraphics[width=1\columnwidth]{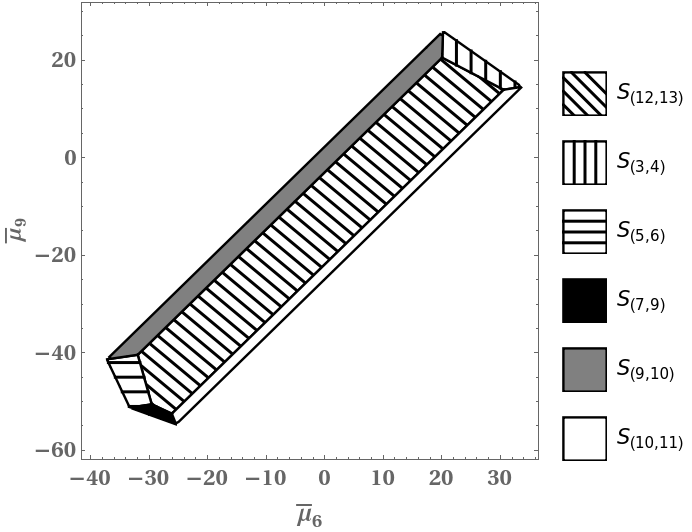}  
               \caption{\footnotesize{Partition of $\Rdet$.}}
               \label{fig:14bus apartc}
      \end{subfigure}
         \begin{subfigure}[t]{0.40\columnwidth}
    \includegraphics[width=1\columnwidth]{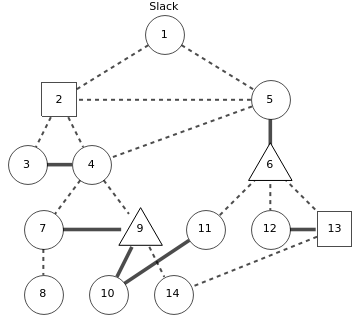} 
          \caption{\footnotesize{Network topology.}}
          \label{fig:14bus apartd}
    \label{fig:2}
  \end{subfigure}
   \caption{\footnotesize{ a,b) Capacity regions for the IEEE $14$-bus network, depicted using different mesh styles, for two different target probabilities;
   c) Subdivision of $\Rdet$ according to which lines are the most vulnerable, as in Eq.~\eqref{eq:eq:partition}; d) IEEE-14 topology. Stochastic and (deterministic) controllable nodes are represented with square and triangular vertexes, respectively. The six solid lines are the most vulnerable ones.}}
\label{fig:14bus apart}

\end{figure}

\subsection{IEEE-\texorpdfstring{$118$}{} test network}
\label{ss: 118}
In this section we perform a case study on a larger system, corresponding to the test case \textit{c118swf.m}~\cite{Murillo2013}.  
The system has $118$ nodes, $210$ lines and $52$ generators, $11$ of which are modeled as wind units (indexed by $j_1,\ldots,j_{11}$).
In order to simulate a more heavily loaded system, we define $I_{\max}$ to be equal to $50\%$ of the line limits provided in the test case.

For our study, we first solve a DC OPF~\cite{Sun2010}, which is an optimization problem determining the generation schedule that minimizes the total system generation cost, while satisfying demand/supply balance and network physical constraints, under the assumptions of the DC approximations.
Let $\bar{\mu}\in\R^{118}$ be the resulting optimal net power injections vector.

Next, we model the $11$ wind generators as OU processes\ignore{(Eq.~\eqref{eq: OU current expl});}, using the hour as the unit for temporal quantities. The parameter $\mu_{k}$ of generator $j_k$ is set to be equal to $\bar{\mu}_{j_k}$, which is interpreted as the \textit{nominal power injection} of generator $j_k$.

The parameters $\eps$, $D=\gamma  I, L= \diag(\{l_i\})$ and $T$ are calibrated in such a way that the standard deviation of each OU process at the end time $T$
matches realistic values for wind power forecasting error (expressed as a fraction of the wind plant installed capacity) over different control periods:
\begin{equation}
\label{eq: OU std}
\text{std}_{j_k}(T)=\sqrt{\frac{\eps l_k^2}{2\gamma}(1-e^{-2T\gamma})}=q(T)\cdot  \mu_{j_k}^{\text{(installed)}}.
\end{equation}
Given $T$, we set $q=q(T),\gamma=1,\eps=1$ and solve Eq.~\ref{eq: OU std} for $l_k$. The values for $q(T)$, shown in Fig.~\ref{fig:118bus_15},\ref{fig:118bus_60}, are taken from~\cite{Hodge2011}, and correspond to the Root Mean Squared Forecast Error obtained applying a persistence forecast to ERCOT wind data. Note that this setting can capture renewable generators with different installed capacities. The overload probabilities are chosen in the range $[10^{-7},10^{-1}]$, and $\tau=0.5$ hours.

To quantify the capacity gain achieved by the different regions, for
each choice of the parameters we solve three distinct DC OPFs, each
incorporating a different capacity region $\mathcal{R}$ in the
constraints. Note that since the capacity regions are convex
polytopes, solving these OPFs has the same computational cost as
solving the deterministic one.

Next, we compare the total system costs, which is the value of the objective function at optimality, to the cost obtained by solving the deterministic OPF (that is, the one incorporating $\Rdet$ in the constraints), by means of the \textit{Cost of Uncertainty} ($\text{CoU}$) metric
\begin{align}\label{eq:CoU}
&\text{CoU}^{(\mathcal{R})}(q,p)=\frac{\text{cost}^{\mathcal{R}}(q,p)-\text{cost}^{\text{det}}}{\text{cost}^{\text{det}}}\ge 0,
\end{align}
defined as the relative increase in system costs when uncertainty-aware reliability constraints are considered.
Figure~\ref{fig:118bus} reports $\text{CoU}(q,p)$ for various values of $q$ and $p$.


\begin{figure}[h!]
  \begin{subfigure}[t]{0.45\columnwidth}
    \includegraphics[width=1.1\columnwidth]{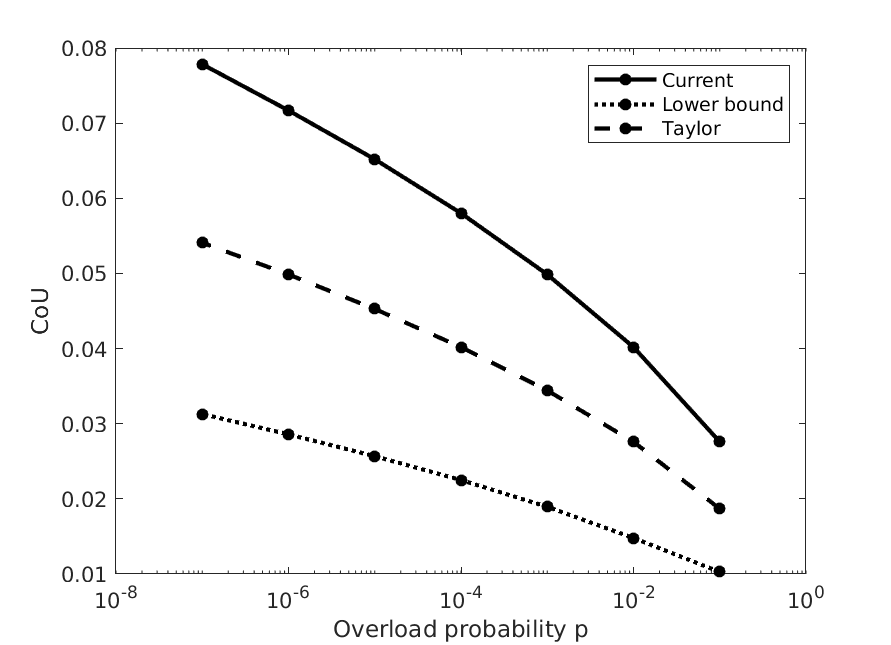} 
     \caption{\footnotesize{$T=1/4$ ($15$ minutes), $q(T)=0.018$.}}
\label{fig:118bus_15}
  \end{subfigure}
  \hspace{0.2cm}
\begin{subfigure}[t]{0.45\columnwidth}
          \includegraphics[width=1.1\columnwidth]{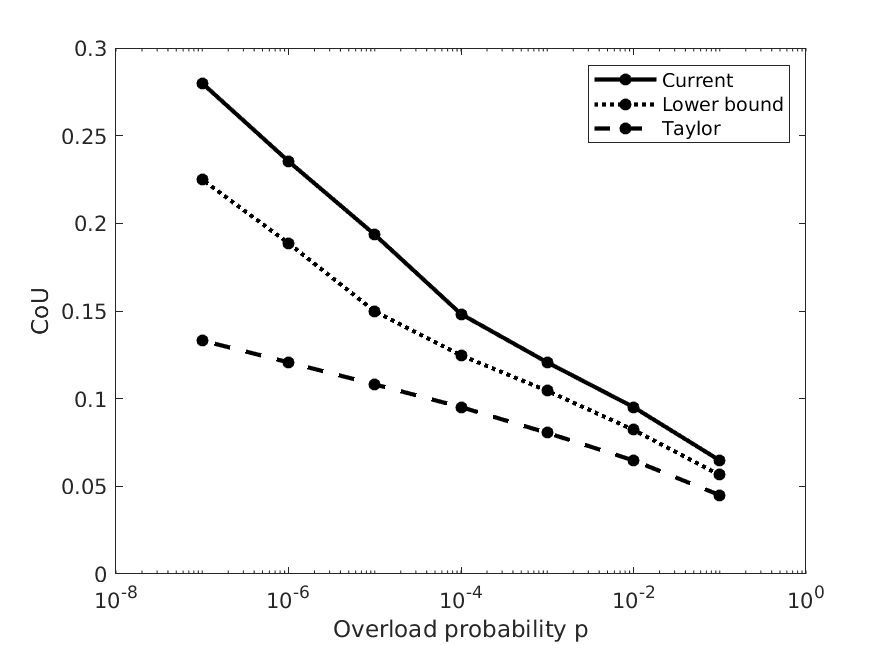}  
          \caption{\footnotesize{$T=1$ ($60$ minutes), $q(T)=0.04$.}}
\label{fig:118bus_60}
  \end{subfigure}
   \caption{\footnotesize{$\text{CoU}^{(\mathcal{R})}$ for different overload probabilities $p$ and time intervals $T$,  $\mathcal{R}\in\{\Rcur,\Rlb,\Rtaylor\}$.}
\label{fig:118bus}
}
\end{figure}

We see that enforcing constraints on line currents results in higher system costs than the ones achieved by using temperature-based constraints, consistently across different probability levels and time intervals.  The gain is more pronounced over shorter intervals, capturing the intuition that current overloads are permissible for short periods, and for smaller probabilities: for instance, $\text{CoU}$ drops from $8\%$ to $3\%$ when $\Rlb$ is used over $\Rcur$, for $T=1/4, p=10^{-7}$.

\section{Concluding remarks}
\label{sc: conclusions}

We employed large deviations theory to develop tractable capacity
regions for power grids with variable power injections, modeled as
small-noise diffusion processes, assuming currents behave according to
the DC power flow equations.  These capacity regions define the set of
initial power injections such that the probability of a
current/temperature overload in a given interval is very small, and
can be used as computationally tractable chance-constraints in OPF
formulations.  Incorporating the transient relationship between line
temperature and line current leads to enlarged capacity regions. While
this enlarged region is difficult to compute, we develop tractable
approximations that improve upon the capacity region defined by the
conservative current overload constraint.

A natural (and possibly straighforward) follow-up to this work would be to consider other linearizations of the AC power flow equations (for example, a first-order Taylor expansion around a nominal operating point). Extensions to more general classes of power injection processes presents another interesting avenue for future work, as does the complementary task of fitting a suitable SDE model based on empirical data of renewable generation.

Finally, we note that potential of our large-deviations results goes beyond the development of capacity regions.  Our results can be used to speed up more detailed simulations, as in
\cite{Wadman2016}. In addition, the ranking of transmission lines
according to their overload probability makes our techniques
applicable to identify the most vulnerable parts of the network~\cite{Nesti2018}.

  {\small
 \bibliographystyle{IEEEtran}
 \bibliography{IEEEabrv,refs_updated}
 }

\clearpage
\appendix
\begin{proof}[Proof of Lemma \ref{lm: rank C}]
Since $w_{\ell}\ne 0$, $I_{\ell,\max}\ne 0$ for all $\ell=1,\ldots,L$, the matrices $W$ and $\Lambda$ are nonsingular. Following
\cite{Bienstock2013} we see that $\rk \breve{B}=N$, and Lemma $2.2$ in~\cite{Bapat2010}
 guarantees that $\rk A=N$ and $\mbox{Ker}(A)=
\mbox{Span} ((1,\ldots,1)^T)$. Since $W$ is nonsingular, $\rk(W A\breve{B})=\rk(A\breve{B})$, and $\rk(A\breve{B})\le\min(\rk (A),\rk (\breve{B}))=N$. On the other hand, if $x\in\mbox{Ker} (A\breve{B})$ then
$A\breve{B}x=0 \iff \breve{B}x\in\mbox{Ker}(A)=
\mbox{Span} ((1,\ldots,1)^T) \iff  \breve{B}x=0\iff x\in\mbox{Ker}(\breve{B})$,
where in the second implication we used that the first component of $\breve{B}x$ is $0$.
Therefore $\mbox{dim}\mbox{Ker} (A\breve{B})=\mbox{dim}\mbox{Ker}(\breve{B})=1$, 
yielding $\rk(A\breve{B})=N+1-\mbox{dim} \mbox{Ker} (A\breve{B})=N$.
Since $\Lambda$ is nonsingular, the matrix $\overline{C}$ must have $N$ linear independent columns. But its first column is zero (since the first column of $\breve{B}$ is zero), therefore the columns from $2$ to $m\le N$ of $\overline{C}$ are linearly independent, i.e. the
matrix $C$ has full rank $m$.
\end{proof}

\begin{proof}[Proof of Lemma~\ref{lm: convex current}] 
  First notice that a vector $(\mu,\mu_D)$ such that
  $\|\nu\|=\|C\mu+C_D\mu_D\|<1$ belongs to the capacity region
  $\Rcur$, i.e. $\IC^*(\mu,\mu_D)=\inf_{\substack{g\in H^1_{\mu}:\\ \|C_D\mu_D+Cg\|=1}} \IP(g)\ge \epsilon \log(1/p)$,
  if and only if the following implication holds:
  \small
\begin{equation}\label{eq: P1}
 \exists g\in H^1_{\mu} \text{ s.t. }\IP(g)<\epsilon\log(1/p)\implies \|C_D\mu_D+Cg\|<1.
\end{equation}
\normalsize
\noindent Consider two admissible vectors
$(\mu,\mu_D),(\mu,\tilde{\mu}_D)\in\Rcur$,
and let $\lambda\in[0,1]$.  We want to show that $(\mu,\lambda
\mu_D+(1-\lambda))\tilde{\mu}_D\in\Rcur$. To this end, take
$g\in H^1_{\mu}$ to be such that $\IP(g)<\epsilon \log(1/p)$, and let us write
$\|\lambda C_D\mu_D+(1-\lambda)C_D\tilde{\mu}_D+Cg\|=
\|\lambda (C_D\mu_D+Cg)+(1-\lambda)(C_D\tilde{\mu}_D+Cg)\|\le
\lambda \|(C_D\mu_D+Cg)\|+(1-\lambda)\|(C_D\tilde{\mu}_D+Cg)\|<\lambda+(1-\lambda)=1,
$ where we used property $\eqref{eq: P1}$ and the fact that
$(\mu,\mu_D),(\mu,\tilde{\mu}_D)$ are admissible. Therefore, $\lambda
\mu_D+(1-\lambda)\tilde{\mu}_D$ is admissible (notice that the above calculation implies in particular that
$\|C\mu+C_D(\lambda\mu_D+(1-\lambda)\tilde{\mu}_D\|<1$).
\end{proof}

\begin{proof}[Proof of Lemma~\ref{lm: overflow_boundary}]
   Define $S_1 = \{g \in H_{\mu}^1\ :\ \|{C_{\ell} g(t)+y_{\ell}}\|_{\infty} \geq 1\},\,
    S_2 = \{g \in H_{\mu}^1\ :\ |C_{\ell} g(T)+y_{\ell}| = 1\}.$
 We have to prove that $\inf_{g \in S_1} \IP(g) = \inf_{g \in S_2} \IP(g).$ Since $S_2 \subset S_1,$ it follows that $\inf_{g \in S_1} \IP(g) \leq \inf_{g \in S_2} \IP(g).$ To prove the reverse inequality, we show that for
  any $g \in S_1,$ there exists $\tilde{g} \in S_2$ such that
  $\IP(\tilde{g}) \leq \IP(g).$
Pick $g \in S_1.$ Let $t' \in [0,T]$ be the first time such that
  $|y_{\ell} + C_{\ell} g(t')| = 1$. Clearly $t'>0$, since $|y_{\ell} + C_{\ell} g(0)|=|\nu_{\ell}|<1$. If $t' = T,$ we may take $\tilde{g}(t) =
  g(t).$ If $t' < T,$ define $\tilde{g}(t)$ by time-shifting $g(t)$ to
  the right as follows:
  \begin{equation*}
    \tilde{g}(t) = \left\{
      \begin{array}{cl}
        \mu & \text{ for } 0 \leq t < T-t', \\
        g(t-T+t') & \text{ for } T-t' \leq t \leq T.
      \end{array}
    \right.
  \end{equation*}
  It is easy to check that $\tilde{g} \in S_2,$ and that
  $\IP(\tilde{g}) \leq \IP(g)$, because the path $\tilde{g}$ incurs no cost up to time $T-t'$. Indeed, since in the interval $[0,T-t']$ $\tilde{g}$ is constantly equal to $\mu$, we have $b(\tilde{g}(t))=b(\mu)=0$ and  $\tilde{g}\pr(t)=0$, yielding
   $\int_{0}^{T-t'} \Bigl(\frac{g\pr_i-b_i(g_i)}{l_i(g_i)}\Bigr)^2dt=0$
   and thus
   {\footnotesize
 \begin{align*}
  \IP(\tilde{g})=
   \int_{T-t'}^{T} \Bigl(\frac{\tilde{g}\pr_i-b_i(\tilde{g}_i)}{l_i(\tilde{g}_i)}\Bigr)^2dt=
    \int_{0}^{t'} \Bigl(\frac{g\pr_i-b_i(g_i)}{l_i(g_i)}\Bigr)^2dt\le \IP(g).
    \end{align*}
    }
 \end{proof}
\begin{lemma}\label{lm: monotonicity}
   The function $a\to\psi_{\ell}^{(a)}$ is non-decreasing for $a >
   \nu_{\ell}$ and non-increasing for $a<\nu_{\ell}$.
 \end{lemma}
 \begin{proof}
  First suppose $a\ge \tilde{a}> \nu_{\ell}\ge 0$. The case
   $a\le \tilde{a}< \nu_{\ell}\le 0$ is analogous. We want to show that for all $f\in y+CH^1_{\mu}$ such that
   $f_{\ell}(T)=a$, there exist a $\tilde{f}\in y+C H^1_{\mu}$ with $\tilde{f}_{\ell}(T)=\tilde{a}$ and
   $\IC(\tilde{f})\le \IC(f)$. Since $f_{\ell}(0)=\nu_l<\tilde{a}\le a$ and $f$ is continuous, there exist a $t'\in(0,T)$
   such that $f(t')=\tilde{a}$. Define $\tilde{f}(t)$ as follows:
  \begin{equation*}
    \tilde{f}(t) = \left\{
      \begin{array}{cl}
        \nu & \text{ for } 0 \leq t < T-t' \\
        f(t-T+t') & \text{ for } T-t' \leq t \leq T
      \end{array}
    \right.
  \end{equation*}
  It is easy to check that $\tilde{f} \in y+CH^1_{\mu}$, $\tilde{f}_{\ell}(T)=\tilde{a}$ and
  $\IC(\tilde{f}) \leq \IC(f).$
  The proof that $\psi_{\ell}^{(a)}$ is non-increasing for $a<\nu_{\ell}$ goes along the same lines.
  \end{proof}
\begin{proof}[Proof of Proposition \ref{pr: OU_regions}]
\ignore{In \cite{Wadman2016} the optimal current paths and values are obtained
for a single line and a deterministic initial value
$y=0\in\mathbb{R}$. Since our focus is on capacity regions, in our
setting we need to consider a general vector $y\in\mathbb{R}^L$ of
deterministic initial values for the currents.  }
Following the
methods in \cite{Wadman2016}, for $\ell\in\mathcal{L\pr}=\{\ell \in\mathcal{L}\,:\, C_{\ell}\ne0\}$ it
can be shown that
\begin{equation}\label{eq: psi_a}
\begin{split}
  \psi_{\ell}^{(a)}=\frac{(a-\nu_{\ell})^2}{C_{\ell}M_TC_{\ell}^{\top}},
  \end{split}
\end{equation}
where $M_t=L^2D^{-1}(I-e^{-2Dt})e^{D(t-T)}$.
The corresponding optimal paths for power injections and currents leading to the overload of line $\ell$ are
\begin{equation}
\begin{split}\label{eq: current path OU}
& X^{({\ell})}(t)=(a-\nu_{\ell})\frac{M_tC_{\ell}^{\top}}{C_{\ell}M_TC_{\ell}^{\top}}+\mu\in\R^{m},\\
& Y^{({\ell})}(t)=CX^{({\ell})}(t)+y\in\R^L.
 \end{split}
\end{equation}
%
 It follows easily
that \begin{equation}
\IC^*(\overline{\mu})=\min_{\ell\in\mathcal{L\pr}}\frac{(1-|\nu_{\ell}|)^2}{C_{\ell}M_TC_{\ell}^{\top}}.
\end{equation}
\ignore{Recall that $\overline{\mu}=[\mu,\mu_D]\in\R^N$ and
$\nu=C\mu+C_D\mu_D\in\R^L$, and define
\[\beta_{\ell}:=\sqrt{\frac{(1-e^{-2\gamma T})\eps\log(1/p)\sigma_{\ell}^2}{\gamma}},\quad \sigma_{\ell}^2:=C_{\ell}L^2C_{\ell}^{\top}.\]}
 A straightforward calculation yields the desired result.
\end{proof}
\begin{proof}[Proof of Lemma \ref{lm: convex temperature}] 
First notice that a vector $(\mu,\mu_D)$
such that
 $\|\nu\|=\|C\mu+C_D\mu_D\|<1$ is admissible if and only if the following implication holds:
 {\small
\begin{equation}\label{eq: P2}
 \exists g\in H^1_{\mu}\text{ s.t. } \IP(g)<\epsilon\log(1/p) \implies \|h^{g,\mu,\mu_D}\|_{\infty}<1 ,
\end{equation}
where
}
\begin{align*}&h^{g,\mu,\mu_D}_{\ell}(t):=\xi_{\tau}(y +Cg)=(y_{\ell}+C_{\ell}\mu)^2 e^{-t/\tau} + \\
&\frac 1\tau \int_0^t e^{-(t-s)/\tau} (y_{\ell}+C_{\ell}g(s))^2 ds,\quad y=C_D\mu_D.
\end{align*}
For all $\ell\in\mathcal{L}$ and for all $t\in[0,T]$, $h^{g,\mu,\mu_D}_\ell(t)$ is non-negative and convex in
$\mu_D$. Using the property \eqref{eq: P2}, the rest of the proof goes along the lines of the proof of lemma \ref{lm: convex current}.
\end{proof}
\begin{proof}[Proof of lemma \ref{lm: temp_bound}] 
The proof follows easily from the observation that the event
  $\norm{\Theta^{\eps,\tau}_{\ell}} \geq 1$ implies the event
  $\norm{Y^{\epsilon}_{\ell}} \geq \alpha_{\ell}.$ Indeed, it is easy to check
  that if $|Y_{\ell}^{\epsilon}(t)| < \alpha_{\ell}$ for all $t \in [0,T],$ then it follow from \eqref{eq: ode} that
  $\Theta_{\ell}^{\eps,\tau}(t) < 1$ for all $t \in [0,1].$ Thus, we have

 \begingroup
\allowdisplaybreaks
{\small
  \begin{align*} \omega_{\ell} &= \lim_{\epsilon \da 0} -\epsilon \log
    \prob{\norm{\Theta_{\ell}^{\eps,\tau}} \geq 1}
    \geq \lim_{\epsilon \da 0} -\epsilon \log
    \prob{\norm{Y_{\ell}^{\epsilon}} \geq \alpha_{\ell}} \\
    &= \inf_{g:\ \norm {y_{\ell} + C_{\ell} g}  \geq \alpha_{\ell}} \IP(g)
    = \inf_{g:\ |y_{\ell} + C_{\ell} g(T)| = \alpha_{\ell}} \IP(g)\\
    &= \psi_{\ell}^{(\alpha_{\ell})}\wedge\psi_{\ell}^{(-\alpha_{\ell})}.
  \end{align*}
  }
  \endgroup
  \end{proof}

\begin{proof}[Proof of Proposition~\ref{pr: LB}]
 Thanks to Lemma~\ref{lm: temp_bound}
  we see that $\IT^{(LB)}$
is a lower bound for the temperature decay rate, i.e. $\IT^*\ge \IT^{(LB)}$.
Since $\alpha_{\ell}>1>|\nu_{\ell}|$ $\forall \ell$, Lemma~\ref{lm: monotonicity}
   implies $\psi_{\ell}^{(\alpha_{\ell})}\wedge\psi_{\ell}^{(-\alpha_{\ell})}\ge
  \psi_{\ell}^{(1)}\wedge\psi_{\ell}^{(-1)},$
yielding $\IT^{(LB)}\ge\IC^*$.
 \end{proof}

 \begin{proof}[Proof of Proposition~\ref{pr: delta,eta}]
  Thanks to Lemma~\ref{lm: temp_bound}
   we have $\IT^{(LB)}=\min_{{\ell}\in\mathcal{L\pr}}\psi_{\ell}^{(\alpha_{\ell})}\wedge\psi_{\ell}^{(-\alpha_{\ell})}$.
From Eq.~\eqref{eq: psi_a}
 we get
$\psi_{\ell}^{(\alpha_{\ell})}=\frac{(\alpha_{\ell}-\nu_{\ell})^2}{C_{\ell}M_TC_{\ell}^T}$
and thus $\psi_{\ell}^{(\alpha_{\ell})}\wedge\psi_{\ell}^{(-\alpha_{\ell})}=
\psi_{\ell}^{(\text{sign}(\nu_{\ell})\alpha_{\ell})}=\frac{(\alpha_{\ell}-|\nu_{\ell}|)^2}{C_{\ell}M_TC_{\ell}^T}$, yielding the expression for $ \Rlb$.
In the case $D=\gamma I$, a straightforward calculation yields the result.\end{proof}
 \begin{proof}[Proof of Proposition~\ref{pr: OU_taylor}] 
In the case $D=\gamma I$, according to equation \eqref{eq: current path OU}, the optimal current paths to overflow in line ${\ell}$ and
the corresponding decay rate are
\begin{equation}
\begin{split}
&Y^{({\ell})}(t)=(\text{sign}(\nu_{\ell})-\nu_{\ell})\frac{(1-e^{-2\gamma t})e^{\gamma(t-T)}}{1-e^{2\gamma T}}R^{\ell}+\nu,\\
&\psi_{\ell}=\frac{\gamma}{1-e^{-2\gamma T}}\frac{(1-|\nu_{\ell}|)^2}{\sigma_{\ell}^2},
\end{split}
\end{equation}
 where
 $R^{\ell}:=\frac{CL^2C_{\ell}^T}{C_{\ell}L^2C_{\ell}^T}\in\mathbb{R}^L$ and $\sigma_{\ell}^2=C_{\ell}L^2C_{\ell}^T$. Take any ${\ell}^*\in\argmin_{{\ell}\in\mathcal{L\pr}} \psi_{\ell}$.
Recall that ${\ell}^*$ depends
on the initial condition $\overline{\mu}$, i.e.
${\ell}^*={\ell}^*(\overline{\mu})$.
Letting $S^*=\text{sign}(\nu_{{\ell}^*})-\nu_{{\ell}^*}\in\mathbb{R}$ and
$R^*=R^{\ell^*}$, the optimal current path to overflow is
$f_*(t)=Y^{({\ell}^*)}(t)$
and in particular
$f_*(0)=\nu,
 f_*(T)=S^*R^*+\nu,
 (f_*)\pr(0)=\frac{2\gamma e^{\gamma T}}{1-e^{-2\gamma T}}S^*R^*,
 (f_*)\pr(T)=\frac{\gamma(1+e^{-2\gamma T})}{1-e^{-2\gamma T}}S^*R^*.$
After a lengthy but straightforward calculation, which is reported below,
the formula for the Taylor approximation reads
\begin{equation}
\IT^{(TL)}(\bar{\mu})=(1+2\tau_0\gamma)\IC^*(\bar{\mu}).
\end{equation}
The capacity region defined by the Taylor approximation is 
  {\small
  \begin{align*}
   &\Rtaylor=\\
  =&\bigcap_{{\ell}\in\mathcal{L\pr}}\{\overline{\mu}\in\mathbb{R}^N:
\frac{\gamma (1-|\nu_{\ell}|)^2}{(1-e^{-2\gamma T})\sigma_{\ell}^2}
  (1+2\tau_0\gamma)>-\eps\log(p)\},
  \end{align*}
  }
which can be rewritten as
\begin{equation*}
\begin{split}
    \Rtaylor
    =&\bigcap_{{\ell}\in\mathcal{L\pr}}\Bigl\{\overline{\mu}\in\mathbb{R}^N\,:\,|\nu_{\ell}|<1-\eta_{\ell}/\sqrt{1+2\tau_0\gamma}\Bigr\}.
\end{split}\end{equation*}
\end{proof}

\begin{proof}[Proof of equation \eqref{eq:calculationstaylor}]
Since $\xi_{\tau}(f)=h$ if and only if
$\tau h\pr+h=f^2$, the temperature rate function reads
{\small \begin{align*}
&\IT(h)= \begin{cases}
G(\tau,h)\quad &\text{if } h\in \xi_{\tau}(y+C H_{\mu}^1),\\
\infty &\text{otherwise},
\end{cases}\\
&G(\tau,h)=\ICS(\tau h\pr +h)=\IC(f_{\tau h\pr+ h})=\IP(C^+(f_{\tau h\pr+ h}-y)),
 \end{align*}}
 
\noindent where
$
 \ICS(F)=\inf_{\substack{_{f\in H^1_{\nu}}:\, f^2=F}} \IC(f)
$
is the rate function for the current squared process $(Y^\eps(t))^2$
and $f_F:=\argmin_{\substack{_{f\in H^1_{\nu}}:\\ f^2=F}} \IC(f).$
Note that $\ICS(F)$ can be written as
\begin{align*}
& \ICS(F)=\sum_{i=1}^m \int_{0}^T \WH(F(t),F\pr(t))dt,\\
&\WH(F(t),F\pr(t))=\frac{1}{2}
\Bigr[\frac{C_i^+ f_{F}\pr(t)-b_i(C_i^+(f_{F}(t)-y))}
{l_i(C_i^+(f_{F}(t)-y))}\Bigl]^2.
\end{align*}

The partial derivatives of the function
\[\tau\to \WH((\tau h\pr+h),(\tau h\prpr+h\pr))\] in $\tau=0$ read
{\small
\begin{equation*}
            \pa\WH\Bigl(\tau h\pr+h,\tau h\prpr+h\pr\Bigr)\Bigl|_{\tau=0}=\WH^{(\ell)}(h,h\pr)h\pr_{\ell}+
  \WH^{(L+\ell)}(h,h\pr)h\prpr_{\ell},
  \end{equation*}
  }
yielding
\begingroup
\allowdisplaybreaks
\begin{align*}
&  \sum_{\ell=1}^L \pa\WH\Bigl(\tau h\pr+h,\tau h\prpr+h\pr \Bigr)\Bigl|_{\tau=0}=\frac{d}{dt}\WH(h,h\pr),\\
&\sum_{\ell=1}^L \pa G(\tau,h)|_{\tau=0}
  =\sum_{\ell=1}^L \pa \ICS\Bigl(\tau h\pr+h\Bigr)\Bigr|_{\tau=0}\\
  =&\sum_{\ell=1}^L\sum_{i=1}^m \int_{0}^T  \pa\WH\Bigl(\tau h\pr+h,\tau h\prpr+h\pr\Bigr)\Bigl|_{\tau=0}\\
  =&\sum_{i=1}^m\int_{0}^T \sum_{\ell=1}^L \pa\WH\Bigl(\tau h\pr+h,\tau h\prpr+h\pr\Bigr)\Bigl|_{\tau=0}\\
  =&\sum_{i=1}^m\int_{0}^T \frac{d}{dt}\WH(h,h\pr) dt=\sum_{i=1}^m\Bigl[ \WH(h(T),h\pr(T))-\\&\WH(h(0),h\pr(0))\Bigr]
  =:\Phi(f_{h}(0),f_{h}(T),f_{h\pr}(0),f_{h\pr}(T))=:\Phi_{f_{h}}.
  \end{align*}
  \endgroup
   If $\tau=\tau_0(1,\ldots,1)^T$, $\tau_0>0$, we get
$\tau\cdot \nabla G(\tau,h)|_{\tau=0}=\tau_0\Phi_{f_{h}}$.
Finally, formula \eqref{eq:closed_form} follows by noticing that if $f_*$ is the optimal current path and $h_*=(f_*)^2$ then $f_{h_*}=f_*$.
\end{proof}

\begin{proof}[Proof of equation \eqref{eq: taylor first order OU}]
 We have
 \begingroup
\allowdisplaybreaks
 \small
\begin{flalign*}
  &\WH(h_*(T),(h_*)\pr(T))=H_i(C_i^{+}(f_*(T)-y),C_i^{+}f_*\pr(T))\\
   =&\frac{1}{2l_i^2} \Bigl(C_i^{+}f_*\pr(T)+\gamma C_i^{+}(f_*(T)-y)-\gamma\mu_i\Bigr)^2\\
  =&\frac{1}{2l_i^2} \Bigl(\frac{\gamma(1+e^{-2\gamma T})}{1-e^{-2\gamma T}}S^*C_i^{+}R+\gamma(C_i^{+}S^*R^*+C_i^{+}(\nu-y))-\gamma\mu_i\Bigr)^2\\
  =&\frac{\Bigl(\gamma(1+e^{-2\gamma T})S^*C_i^{+}R^*+(1-e^{-2\gamma T})\gamma(C_i^{+}S^*R^*)\Bigr)^2}{2l_i^2 (1-e^{-2\gamma T})^2} \\
  =&\frac{2\gamma^2(1-|\nu_{{\ell}^*}|)^2}{l_i^2 (1-e^{-2\gamma T})^2} \Bigl(C_i^{+}R^*\Bigr)^2;\\
&\WH(h_*(0),(h_*)\pr(0))=H_i(C_i^{+}(f_*(0)-y),C_i^{+}f_*\pr(0))\\
   =&\frac{1}{2l_i^2} \Bigl(C_i^{+}f_*\pr(0)+\gamma C_i^{+}(f_*(0)-y)-\gamma\mu_i\Bigr)^2\\
 =&\frac{1}{2l_i^2} \Bigl(\frac{2\gamma e^{-\gamma T}}{1-e^{-2\gamma T}}S^*C_i^{+}R^*+\gamma C_i^{+}(\nu-y)-\gamma\mu_i\Bigr)^2\\
   =&\frac{2\gamma^2e^{-2\gamma T}(1-|\nu_{{\ell}^*}|)^2}{l_i^2 (1-e^{-2\gamma T})^2} \Bigl(C_i^{+}R^*\Bigr)^2;
 \end{flalign*}
 \endgroup
\begin{align*}\Phi_{f_*}=&\sum_{i=1}^m \Bigl[
  \WH(h_*(T),(h_*)\pr(T))-
  \WH(h_*(0),(f_*^2)\pr(0))\Bigr]\\
 &=\frac{2\gamma^2(1-|\nu_{{\ell}^*}|)^2}{1-e^{-2\gamma t }} \sum_{i=1}^m \Bigl(\frac{C_i^{+}R^*}{l_i}\Bigr)^2. \end{align*}
The Taylor approximation thus reads
\begingroup
\allowdisplaybreaks
\small
\begin{equation*}\label{eq: expansion_OU}
  \begin{split}
  &\IT^{(TL)}(\tau_0,\overline{\mu}):=\IC^*(\overline{\mu})+\tau_0\Phi_h=\\
&\frac{\gamma}{1-e^{-2\gamma T}}\frac{ (1-|\nu_{{\ell}^*}|)^2}{\sigma_{{\ell}^*}^2} + \frac{2\tau_0\gamma^2(1-|\nu_{{\ell}^*}|)^2}{1-e^{-2\gamma t }} \sum_{i=1}^m \Bigl(\frac{C_i^{+}R^*}{l_i}\Bigr)^2=\\
 &\frac{\gamma (1-|\nu_{{\ell}^*}|)^2}{1-e^{-2\gamma T}} \Bigl(\frac{1} {\sigma_{{\ell}^*}^2}+2\tau_0\gamma \sum_{i=1}^m \Bigl(\frac{C_i^{+}R^*}{l_i}\Bigr)^2  \Bigr)=\\
&\frac{\gamma (1-|\nu_{{\ell}^*}|)^2}{(1-e^{-2\gamma T})\sigma_{{\ell}^*}^2}
 \Bigl( 1+2\tau_0\gamma \sum_{i=1}^m \frac{l_i^2C_{{\ell}^*i}^2}{\sigma_{{\ell}^*}^2}
 \Bigr)=\\
  &\frac{\gamma (1-|\nu_{{\ell}^*}|)^2}{(1-e^{-2\gamma T})\sigma_{{\ell}^*}^2}
 \Bigl( 1+2\tau_0\gamma  \Bigr)=
 ( 1+2\tau_0\gamma)\IC^*(\overline{\mu}).
\end{split}
\end{equation*}
\endgroup
\ignore{\begingroup
\allowdisplaybreaks
\begin{equation*}\label{eq: expansion_OU}
  \begin{split}
  &\IT^{(TL)}(\tau_0,\overline{\mu}):=\IC^*(\overline{\mu})+\tau_0\Phi_h=\\
&\frac{\gamma}{1-e^{-2\gamma T}}\frac{ (1-|\nu_{{\ell}^*}|)^2}{\sigma_{{\ell}^*}^2} + \frac{2\tau_0\gamma^2(1-|\nu_{{\ell}^*}|)^2}{1-e^{-2\gamma t }} \sum_{i=1}^m \Bigl(\frac{C_i^{+}R^*}{l_i}\Bigr)^2=\\
 &\frac{\gamma (1-|\nu_{{\ell}^*}|)^2}{1-e^{-2\gamma T}} \Bigl(\frac{1} {\sigma_{{\ell}^*}^2}+2\tau_0\gamma \sum_{i=1}^m \Bigl(\frac{C_i^{+}R^*}{l_i}\Bigr)^2  \Bigr)=\\
 &\frac{\gamma (1-|\nu_{{\ell}^*}|)^2}{(1-e^{-2\gamma T})\sigma_{{\ell}^*}^2}
 \Bigl( 1+2\tau_0\gamma \sum_{i=1}^m \Bigl(\frac{C_i^{+}R^*}{l_i}\Bigr)^2 \sigma_{{\ell}^*}^2\Bigr)=\\
&\frac{\gamma (1-|\nu_{{\ell}^*}|)^2}{(1-e^{-2\gamma T})\sigma_{{\ell}^*}^2}
 \Bigl( 1+2\tau_0\gamma \sum_{i=1}^m \frac{l_i^2C_{{\ell}^*i}^2}{\sigma_{{\ell}^*}^2}
 \Bigr)=\\
  &\frac{\gamma (1-|\nu_{{\ell}^*}|)^2}{(1-e^{-2\gamma T})\sigma_{{\ell}^*}^2}
 \Bigl( 1+2\tau_0\gamma  \Bigr)=
 ( 1+2\tau_0\gamma)\IC^*(\overline{\mu}).
\end{split}
\end{equation*}
\endgroup}
\end{proof}

\end{document}